\documentclass[11pt, reqno]{amsart}
\usepackage{amssymb, amsmath, amsthm, mathrsfs, amsfonts, latexsym, tikz}
\usepackage[margin=1.2in]{geometry}

\newtheorem{theorem}{Theorem}[section]
\newtheorem{proposition}[theorem]{Proposition}
\newtheorem{lemma}[theorem]{Lemma}

\def\R{{\mathbb R}}

\def\eps{\varepsilon}

\def\E{{\mathbb E}}
\def\P{{\mathbb P}}

\numberwithin{equation}{section}

\begin{document}

\title[Propagation of Singularities for Stochastic Wave Equation]
{Propagation of Singularities for the Stochastic\\ Wave Equation}
\author{Cheuk Yin Lee
 \and Yimin Xiao
%\thanks{Research of Y. Xiao is partially supported by NSF grants DMS-1607089.}
}

\keywords{Stochastic wave equation, singularities, law of the iterated logarithm}

\subjclass[2010]{60H15, 60G17}
% stochastic partial differential equations, sample path properties

\begin{abstract}
We study the existence and propagation of singularities of the solution to a 
one-dimensional linear stochastic wave equation driven by an additive Gaussian noise that is 
white in time and colored in space. Our approach is based on a simultaneous law of 
the iterated logarithm and general methods for Gaussian processes.
\end{abstract}

\maketitle

\section{Introduction}

Consider the linear stochastic wave equation in one spatial dimension:
\begin{equation}\label{SWE}
\begin{cases}
\vspace{6pt}
\displaystyle{\frac{\partial^2}{\partial t^2} u(t, x) - \frac{\partial^2}{\partial x^2} u(t, x) = 
\dot{W}(t, x), \quad t \ge 0, x \in \mathbb{R},}\\
\displaystyle{u(0, x) = 0, \quad \frac{\partial}{\partial t} u(0, x) = 0,}
\end{cases}
\end{equation}
where $\dot{W}$ is a Gaussian noise that is white in time and colored in space with
spatial covariance given by the Riesz kernel of exponent $0 < \beta < 1$, i.e.
\begin{equation}\label{Eq:Riesz}
\E[\dot{W}(t, x) \dot{W}(s, y)] = \delta_0(t-s) |x-y|^{-\beta}.
\end{equation}
The existence of the mild solution $u(t, x)$ of \eqref{SWE} was studied by Dalang \cite{D99}, 
see \eqref{SSWE}  below. The purpose of this article is to study the existence and propagation 
of singularities of $\{u(t, x), (t, x) \in \R_+\times \R\} $. 

In this article, %A point $(t, x) \in \R_+\times \R$ is called a singular point of  $u(t, x)$ if 
singularities are associated with the law of the iterated logarithm (LIL), 
or the local modulus of continuity. They refer to the random points at which the process exhibits
local oscillations that are much larger than those given by the LIL. For Brownian motion, 
this phenomenon was first studied by Orey and Taylor \cite{OT74}. It is well known that at a 
fixed time point, the local oscillation of a Brownian sample path satisfies the LIL almost surely. 
However, it is not true that the LIL holds simultaneously for all points with probability 1. Indeed, 
Orey and Taylor \cite{OT74} show that, according to L\'{e}vy's uniform modulus of continuity, 
one can find random points at which the LIL fails and the oscillation is exceptionally large.
Therefore these exceptional points may be defined as singularities. Similarly, singularities can 
be defined for other stochastic processes and, more generally, random fields.

The singularities of the Brownian sheet and the one-dimensional stochastic wave equation
driven by the space-time white noise were studied by Walsh \cite{W82, W}, and the case of 
semi-fractional Brownian sheet was studied by Blath and Martin \cite{BM08}.
Based on a simultaneous law of the iterated logarithm, Walsh \cite{W82} showed that the
singularities of the Brownian sheet propagate parallel to the coordinate axis. Moreover, 
Walsh \cite{W} established an interesting relation between the Brownian sheet and the 
solution $u(t, x)$ to \eqref{SWE} driven by the space-time white noise. 
Indeed, it follows from Theorem 3.1 in \cite{W} 
that the solution can be decomposed into three components:
\begin{equation}\label{Eq:Walsh}
u(t, x) = \frac{1}{2}\left[ B\Big( \frac{t-x}{\sqrt{2}}, \frac{t+x}{\sqrt{2}} \Big) +
\hat{W}\Big(\frac{t-x}{\sqrt{2}}, 0\Big) + \hat{W}\Big(0, \frac{t+x}{\sqrt{2}} \Big) \right],
\end{equation}
where the main component $B$ is a Brownian sheet and $\hat{W}$ is the modified Brownian 
sheet defined in Chapter 1 of Walsh \cite{W}, and  the processes $\{B(s, t), s, t \ge 0\}$, 
$\{\hat{W}(s, 0), s \ge 0\}$ and $\{\hat{W}(0, t), t \ge 0\}$ are independent. This relation 
implies that the singularities of $u(t, x)$ propagate along the characteristic 
lines $t-x = c$ and $t+x = c$.

The study of singularities of the solution to \eqref{SWE} driven by the space-time white noise in 
\cite{W82, W} was later extended by Carmona and Nualart \cite{CN88} to one-dimensional 
nonlinear stochastic wave equations driven by the space-time white noise on the whole line $\R$ and 
on a bounded interval, respectively. Their method is based 
on the general theory of semimartingales and two-parameter strong martingales. They showed that, 
in the white noise case, their solution $X(t, x)$ has the following important properties:
\begin{itemize}
\item[(i).]\, For any $x \in \R$, $\big\{X( \frac h {\sqrt{2}}, x + \frac h {\sqrt{2}}), h \ge 0\big\} $ is 
a continuous semimartingale (cf, \cite[p.741]{CN88}).
\item[(ii).]\, The increments of $X(t, x)$ over a  certain class of rectangles form a two-parameter 
strong martingale (cf, \cite[p.740]{CN88}).
\end{itemize}
In \cite{CN88},  Carmona and Nualart  proved the law of the iterated logarithm for a semimartingale 
by the LIL of Brownian motion and a time change. They also proved that, for a class of 
two-parameter strong martingales, the law of the iterated logarithm in one variable holds 
simultaneously for all values of the other variable. By applying these results and properties (i) and (ii),
Carmona and Nualart  proved the existence and propagation of singularities of the solution
in \cite[Theorem 3.1]{CN88}. 
%Consequently, by applying  In particular, 

The main objective of this article is to study the existence and propagation of singularities of the 
solution to the linear stochastic wave equation \eqref{SWE} driven by a Gaussian noise that is 
white in time and colored in space with spatial covariance  given by (\ref{Eq:Riesz}) with $0 < 
\beta < 1$. 
In this case, the solution is a mean zero Gaussian random field that shares some 
similarity with the fractional Brownian sheet, but it seems to us that there is not a 
natural relation like \eqref{Eq:Walsh} between the solution and the fractional Brownian sheet. 
Also, by applying the computation in the proofs of Lemmas \ref{Lem2} and \ref{Lem3} below, 
one can verify that, when $0 < \beta < 1$, the solution $u(t, x)$ does not have the aforementioned 
properties (i) and (ii). Consequently, the method in Carmona and Nualart \cite{CN88} %based on two-parameter strong martingales 
can not be applied. Our approach for studying the existence and propagation of singularities
of the solution $u(t, x)$ is based on a simultaneous LIL that is established by applying general 
methods for Gaussian processes. We believe that this approach can be applied in more general settings. 
For example, one may study the singularities of the solutions of linear stochastic wave equations driven 
by more general Gaussian noises or in two spatial dimensions. One may also consider non-vanishing 
initial data and study the effects of both initial data and the Gaussian noise on the existence and propagation 
of singularities. These extensions are non-trivial and go beyond the scope of the present paper. 
%We will and will be studied elsewhere.
% that are produced by the noise rather than the initial conditions, so we only consider zero initial data in \eqref{SWE}.
%We expect that this approach can be applied to the solutions of linear stochastic wave equations driven 
%by more general Gaussian noises or in two spatial dimensions.

The rest of this article is organized as follows. First, we establish a simultaneous LIL in Propositions 
2.1 and 3.1 for the solution of the linear stochastic wave equation \eqref{SWE}. We prove that 
after a rotation, the LIL in one variable holds simultaneously for all values of the other variable. 
The proof consists of two parts. The upper bound for the LIL is proved in Section 2 and the 
corresponding lower bound is proved in Section 3. In Section 4, we define singularity for the solution 
and apply the simultaneous LIL to study the propagation of singularities. The main result Theorem 
\ref{Thm:sing} shows that singularities propagate along the characteristic lines.

It would be interesting to study the existence and propagation of singularities of the solutions of the nonlinear 
stochastic wave equations with a Gaussian noise which is white in time and colored in space (\cite{D99}) 
or stochastic wave equations with a multiplicative space-time homogeneous Gaussian noise (\cite{BQS19}). 
For solving these problems,  the martingale based methods of Carmona and Nualart \cite{CN88} are not 
applicable. We believe that the Gaussian method in this article, together with an appropriate approximation 
argument, would be useful.  However, we have not been able to find a quantitative bound on the  
approximation that is precise enough for our purpose. Some new argument seems to be needed.
%do not expect for the moment to be able to extend the results of our paper to these cases. New arguments seem to be needed.

%\medskip

\section{Simultaneous Law of Iterated Logarithm: Upper Bound}

The noise in \eqref{SWE} is defined as the mean zero Gaussian process $W(\varphi)$ indexed by compactly 
supported smooth functions $\varphi \in C^\infty_c(\R_+ \times \R)$ with covariance function
\begin{align}
\begin{aligned}\label{Eq:noise_cov}
\E[W(\varphi) W(\psi)]
& = \int_{\R_+} ds \int_\R dy \int_\R dy' \, \varphi(s, y) |y-y'|^{-\beta} \psi(s, y')\\
& = \frac{1}{2\pi}\int_{\R_+} ds \int_\R \mu(d\xi)\, \mathscr{F}(\varphi(s, \cdot))(\xi) \overline{\mathscr{F}(\psi(s, \cdot))(\xi)}
\end{aligned}
\end{align}
for all $\varphi, \psi \in C^\infty_c(\R_+ \times \R)$,
where $\mu$ is the measure whose Fourier transform is $|\cdot|^{-\beta}$
and $\mathscr{F}(\varphi(s, \cdot))(\xi)$ is the Fourier transform of the function $y \mapsto \varphi(s, y)$ in the 
following convention:
\[ \mathscr{F}(\varphi(s, \cdot))(\xi) = \int_\R e^{-i\xi y} \varphi(s, y) dy. 
\]
Note that $\mu(d\xi) = C_\beta |\xi|^{-1+\beta} d\xi$, where 
\[ C_\beta = \frac{\pi^{1/2} 2^{1-\beta}\Gamma(\frac{1}{2} - \frac{\beta}{2})}{\Gamma(\frac{\beta}{2})}, \]
see \cite[p.117]{S}. We assume that $W$ is defined on a complete probability space $(\Omega, \mathscr{F}, \P)$.

Following \cite{D99, DF98}, for any bounded Borel set $A$ in $\R_+ \times \R$, we can define 
\[ W(A) = \lim_{n \to \infty} W(\varphi_n) \]
in the sense of $L^2(\P)$-limit, where $(\varphi_n)$ is a sequence in $C^\infty_c(\R_+ \times \R)$ with a compact set $K$
such that $\mathrm{supp}\,\varphi_n \subset K$ for all $n$ and $\varphi_n \to {\bf 1}_A$.
From \eqref{Eq:noise_cov}, it follows that for any bounded Borel sets $A, B$ in $\R_+\times \R$, we have
\begin{align}
\begin{aligned}\label{Eq:noise_cov2}
\E[W(A) W(B)]
& = \int_{\R_+} ds \int_\R dy \int_\R dy' \, {\bf 1}_A(s, y) |y-y'|^{-\beta} {\bf 1}_B(s, y')\\
& = \frac{1}{2\pi} \int_{\R_+} ds \int_\R \frac{C_\beta d\xi}{|\xi|^{1-\beta}}\, \mathscr{F}({\bf 1}_A(s, \cdot))(\xi) 
\overline{\mathscr{F}({\bf 1}_B(s, \cdot))(\xi)}.
\end{aligned}
\end{align}

In dimension one, the fundamental solution of the wave equation is $\frac{1}{2}{\bf 1}_{\{ |x| \le t \}}$, so the mild
solution of \eqref{SWE} is 
\begin{equation}\label{SSWE} 
u(t, x) = \frac{1}{2} \int_0^t \int_\R {\bf 1}_{\{ |x-y| \le t-s \}}(s, y)\, W(ds\,dy) = \frac{1}{2} \, W(\Delta(t, x)), 
\end{equation}
where $\Delta(t, x) = \{(s, y) \in \R_+ \times \R : 0 \le s \le t,\, |x-y| \le t-s\}$. The set $\Delta(t, x)$ is indicated by 
the shaded region in Figure \ref{fig1}.

\begin{figure}
\begin{tikzpicture}
\path [fill=lightgray] (0,-2) -- (0,4) -- (3,1);
\draw [->] (0,-2.5) -- (0,4.6);
\draw [->] (0,0) -- (4,0);
\draw (0, -2) -- (0, 4) -- (3,1) -- (0,-2);
\draw [dashed] (3,0) node [below] {$t$} -- (3,1);
\draw [dashed] (0,1) node [left] {$x$} -- (3,1);
\draw [fill] (3,1) circle [radius=.05];
\draw [fill] (0,4) circle [radius=.05] node [left]{$x+t$};
\draw [fill] (0,-2) circle [radius=.05] node [left]{$x-t$};
\end{tikzpicture}
\caption{}\label{fig1}
\end{figure}

Consider a new coordinate system $(\tau, \lambda)$ obtained by rotating the $(t, x)$-coordinates by $-45^\circ$. 
In other words,
\[ (\tau, \lambda) = \Big(\frac{t-x}{\sqrt{2}}, \frac{t+x}{\sqrt{2}}\Big) \quad \text{and} \quad (t, x) 
= \Big( \frac{\tau+\lambda}{\sqrt{2}}, 
\frac{-\tau+\lambda}{\sqrt{2}} \Big). \]
For $\tau \ge 0$, $\lambda \ge 0$, let us denote 
\[\tilde{u}(\tau, \lambda) = u\Big(\frac{\tau+\lambda}{\sqrt{2}}, \frac{-\tau+\lambda}{\sqrt{2}}\Big).\]

We are going to prove a simultaneous LIL for the Gaussian random field $\{\tilde{u}(\tau, \lambda), \tau \ge 0, \lambda \ge 0 \}$. 
This is divided into Propositions 2.1 and 3.1, where an upper bound and the corresponding lower bound for the LIL in $\lambda$, 
which holds simultaneously for all values of $\tau$, are proved.  By using a symmetry argument, we can also prove that the LIL 
in $\tau$ holds simultaneously for all $\lambda$.

\begin{proposition}\label{Prop:sim_LIL_UB}
For any $\lambda > 0$, we have
\begin{equation} \label{Eq:UB1}
\P\left( \limsup_{h \to 0+} \frac{|\tilde{u}(\tau, \lambda + h) - \tilde{u}(\tau, \lambda)|}{\sqrt{(\tau+\lambda)h^{2-\beta}\log\log(1/h)}}
 \le K_\beta \textup{ for all } \tau \in [0, \infty) \right) = 1,
\end{equation}
where 
\begin{equation}\label{Eq:K} 
K_\beta = \bigg( \frac{2^{(1-\beta)/2}}{{(2-\beta)(1-\beta)}}\bigg)^{1/2}.
\end{equation}
\end{proposition}

In order to prove (\ref{Eq:UB1}), we will need several lemmas.

\begin{lemma}\label{Lem1}
For any $0 < \beta < 1$, $a < b$ and $c < d$, we have
\begin{equation}\label{Eq1}
C_\beta \int_{-\infty}^\infty |\mathscr{F}{\bf 1}_{[a, b]}(\xi)|^2 \frac{d\xi}{|\xi|^{1-\beta}} = \frac{4\pi}{(2-\beta)(1-\beta)}(b-a)^{2-\beta}
\end{equation}
and
\begin{equation}\label{Eq2}
C_\beta \int_{-\infty}^\infty \mathscr{F}{\bf 1}_{[a, b]}(\xi) \overline{\mathscr{F}{\bf 1}_{[c, d]}(\xi)} \frac{d\xi}{|\xi|^{1-\beta}} 
= \frac{2\pi}{(2-\beta)(1-\beta)}\Big( |c-b|^{2-\beta} + |d-a|^{2-\beta} - |c-a|^{2-\beta} - |d-b|^{2-\beta} \Big).
\end{equation}
\end{lemma}

\begin{proof}
The Fourier transform of the function ${\bf 1}_{[a, b]}$ is
\[ \mathscr{F}{\bf 1}_{[a, b]}(\xi) = \frac{e^{-i\xi a} - e^{-i\xi b}}{i\xi}. \]
It follows that
\begin{align*}
C_\beta \int_{-\infty}^\infty |\mathscr{F}{\bf 1}_{[a, b]}(\xi)|^2 \frac{d\xi}{|\xi|^{1-\beta}} 
&= C_\beta \int_{-\infty}^\infty {|e^{i\xi(b-a)}-1|^2}\frac{d\xi}{|\xi|^{3-\beta}}\\
& = C_\beta (b-a)^{2-\beta} \int_{-\infty}^\infty {|e^{i\xi}-1|^2}\frac{d\xi}{|\xi|^{3-\beta}}.
\end{align*}
The last equality follows by scaling.
The proof of Proposition 7.2.8 of \cite{ST} shows that
\[ \int_{-\infty}^\infty |e^{i\xi}-1|^2 \frac{d\xi}{|\xi|^{3-\beta}} = \frac{2\pi}{(2-\beta) \Gamma(2-\beta) \sin(\frac{\pi\beta}{2})}. \]
Also, using the relations $\Gamma(2z) = 2^{2z-1} \pi^{-1/2} \Gamma(z) \Gamma(z+1)$, $\Gamma(z)\Gamma(1-z) 
= \pi/\sin(\pi z)$ and $z\Gamma(z) = \Gamma(z+1)$ (cf. \cite[p.895--896]{GR}), we can show that
\[ C_\beta = \frac{2 \Gamma(2-\beta) \sin(\frac{\pi\beta}{2})}{1-\beta}. \]
Hence \eqref{Eq1} follows.

For the second part,
\begin{align*}
\int_{-\infty}^\infty \mathscr{F}{\bf 1}_{[a, b]}(\xi) \overline{\mathscr{F}{\bf 1}_{[c, d]}(\xi)} \frac{d\xi}{|\xi|^{1-\beta}} 
= \int_{-\infty}^\infty \big(e^{i\xi (c-a)} + e^{i\xi (d-b)} - e^{i\xi (c-b)} - e^{i\xi(d-a)}\big) \frac{d\xi}{|\xi|^{3-\beta}}.
\end{align*}
Note that this integral is real, so we have
\begin{align*}
 \int_{-\infty}^\infty \mathscr{F}{\bf 1}_{[a, b]}(\xi) \overline{\mathscr{F}{\bf 1}_{[c, d]}(\xi)} \frac{d\xi}{|\xi|^{1-\beta}}
& = \frac{1}{2} \int_{-\infty}^\infty \big(e^{i\xi (c-a)} + e^{-i\xi (c-a)} + e^{i\xi (d-b)} + e^{-i\xi (d-b)}\\
& \qquad - e^{i\xi (c-b)} - e^{-i\xi (c-b)} - e^{i\xi(d-a)} - e^{-i\xi (d-a)}\big) \frac{d\xi}{|\xi|^{3-\beta}}.
\end{align*}
Since $|e^{i\xi(x-y)}-1|^2 = 2 - e^{i\xi(x-y)} - e^{-i\xi(x-y)}$ for all $x, y \in \R$, we have
\begin{align*}
& \int_{-\infty}^\infty \mathscr{F}{\bf 1}_{[a, b]}(\xi) \overline{\mathscr{F}{\bf 1}_{[c, d]}(\xi)} \frac{d\xi}{|\xi|^{1-\beta}}\\
& = \frac{1}{2} \int_{-\infty}^\infty \big( -|e^{i\xi (c-a)} - 1|^2 - |e^{i\xi (d-b)}-1|^2 + |e^{i\xi (c-b)} -1|^2 + |e^{i\xi(d-a)}-1|^2 \big)
\frac{d\xi}{|\xi|^{3-\beta}}.
\end{align*}
Hence the result (\ref{Eq2}) follows from the first part of the proof.
\end{proof}

The next two lemmas are  concerned with the variances of the increments of $\tilde u$ between two points and 
over an arbitrary rectangle, respectively. They will be used in the proofs of Propositions 2.1 and 3.1 below.

\begin{lemma}\label{Lem2}
For any $\tau, \lambda, h > 0$,
\[ \E[(\tilde{u}(\tau, \lambda+h) - \tilde{u}(\tau, \lambda))^2] =  \frac{1}{2}\,K_\beta^2 \left[(\tau+\lambda) h^{2-\beta} 
+ (3-\beta)^{-1} h^{3-\beta}\right], 
\]
where $K_\beta $ is the constant defined in (\ref{Eq:K}).
%\begin{equation}\label
%K_\beta = \bigg( \frac{2^{(1-\beta)/2}}{{(2-\beta)(1-\beta)}}\bigg)^{1/2}.
%\end{equation}
\end{lemma}

\begin{proof}
Note that
\begin{align*}
 \E[(\tilde{u}(\tau, \lambda+h) - \tilde{u}(\tau, \lambda))^2]
& = \E\left[\Big({u}\Big(\frac{\tau+\lambda + h}{\sqrt{2}}, \frac{-\tau+\lambda+h}{\sqrt{2}}\Big) - {u}
\Big(\frac{\tau+\lambda}{\sqrt{2}}, \frac{-\tau+\lambda}{\sqrt{2}}\Big)\Big)^2\right]\\
& = \frac{1}{4} \, \E\left[\left(W\Big({\Delta\Big(\frac{\tau+\lambda + h}{\sqrt{2}}, \frac{-\tau+\lambda+h}
{\sqrt{2}}\Big) \big\backslash \Delta\Big(\frac{\tau+\lambda}{\sqrt{2}}, \frac{-\tau+\lambda}{\sqrt{2}}\Big)}\Big)\right)^2\right].
\end{align*}
Then by \eqref{Eq:noise_cov2} and Lemma \ref{Lem1},
\[
\begin{split}
& \E[(\tilde{u}(\tau, \lambda+h) - \tilde{u}(\tau, \lambda))^2]\\
& = \frac{1}{8\pi} \left\{ \int_0^{\frac{\tau+\lambda}{\sqrt{2}}} ds \int_{-\infty}^\infty \frac{C_\beta d\xi}{|\xi|^{1-\beta}} 
\big|\mathscr{F}{\bf 1}_{[\sqrt{2}\lambda -s, \sqrt{2}(\lambda + h) - s]}(\xi)\big|^2 \right. \\
& \qquad \qquad + \left.\int_{\frac{\tau+\lambda}{\sqrt{2}}}^{\frac{\tau+\lambda+h}{\sqrt{2}}} ds \int_{-\infty}^\infty 
\frac{C_\beta d\xi}{|\xi|^{1-\beta}}\big|\mathscr{F}{\bf 1}_{[-\sqrt{2}\tau + s, \sqrt{2}(\lambda + h) - s]}(\xi) \big|^2 \right\}\\
& = \frac{1}{2(2-\beta)(1-\beta)} \left\{\int_0^{\frac{\tau+\lambda}{\sqrt{2}}} \big(\sqrt{2} h\big)^{2-\beta} ds 
+ \int_{\frac{\tau+\lambda}{\sqrt{2}}}^{\frac{\tau+\lambda+h}{\sqrt{2}}} \big(\sqrt{2}(\tau+\lambda+h)-2s\big)^{2-\beta} ds \right\}\\
& = \frac{1}{2}\,K_\beta^2 \left[(\tau+\lambda) h^{2-\beta} + (3-\beta)^{-1} h^{3-\beta}\right].  
\end{split}
\]
This proves Lemma \ref{Lem2}.
\end{proof}

\begin{lemma}\label{Lem3}
Fix $\lambda \ge 0$. 
Then, for any $0 \le \tau \le \tau'$ and $0 \le h \le h'$,
\begin{equation}\label{Eq:rec-inc}
\begin{split}
&\E\big[(\tilde u(\tau', \lambda + h') - \tilde u(\tau', \lambda + h) - \tilde u(\tau, \lambda + h') + \tilde u(\tau, \lambda + h))^2\big]\\
& = \begin{cases}
\medskip
\frac 1 2 K_\beta^2 (h' - h)^{2-\beta} \big[ (\tau'-\tau) - \frac{1-\beta}{3-\beta} (h'-h)\big] & \text{if } h'-h \le \tau'-\tau,\\
\frac 1 2 K_\beta^2 (\tau' - \tau)^{2-\beta} \big[(h'-h) - \frac{1-\beta}{3-\beta} (\tau'-\tau)\big] & \text{if } h'-h > \tau'-\tau.
\end{cases}
\end{split}
\end{equation}
\end{lemma}

\begin{figure}
\begin{tikzpicture}
\draw [->] (0,-3.75) -- (0,4) node [above] {$x$};
\draw [->] (0,0) -- (6.8,0) node [right] {$t$};
\draw [->] (0,0) -- (3.6,3.6) node [above right] {$\lambda$};
\draw [->] (0,0) -- (3.7,-3.7) node [right] {};
\draw [thick,domain=5.5/sqrt(2):8/sqrt(2)] plot (\x, {8/sqrt(2)-\x});
\draw [dashed,domain=4/sqrt(2):5.5/sqrt(2)] plot (\x, {8/sqrt(2)-\x});
\draw [thick,domain=4/sqrt(2):6.5/sqrt(2)] plot (\x, {5/sqrt(2)-\x});
\draw [dashed,domain=2.5/sqrt(2):4/sqrt(2)] plot (\x, {5/sqrt(2)-\x});
\draw [thick,domain=6.5/sqrt(2):8/sqrt(2)] plot (\x, {-8/sqrt(2)+\x});
\draw [dashed,domain=4/sqrt(2):6.5/sqrt(2)] plot (\x, {-8/sqrt(2)+\x});
\draw [thick,domain=4/sqrt(2):5.5/sqrt(2)] plot (\x, {-3/sqrt(2)+\x});
\draw [dashed,domain=1.5/sqrt(2):4/sqrt(2)] plot (\x, {-3/sqrt(2)+\x});
\draw [thick,domain=-0.5/sqrt(2):2.5/sqrt(2)] plot ({5.5/sqrt(2)}, \x);
\draw [thick,domain=-1.5/sqrt(2):1.5/sqrt(2)] plot ({6.5/sqrt(2)}, \x);
\draw [fill] ({4/sqrt(2)} ,{4/sqrt(2)}) circle [radius=.05] 
node [xshift=-0.6cm, yshift=0.25cm]{$\lambda+h'$};
\draw [fill] ({2.5/sqrt(2)},{2.5/sqrt(2)}) circle [radius=.05] 
node [xshift=-0.6cm, yshift=0.25cm]{$\lambda+h$};
\draw [fill] ({1.5/sqrt(2)},{-1.5/sqrt(2)}) circle [radius=.05] node [below left]{$\tau$};
\draw [fill] ({4/sqrt(2)},{-4/sqrt(2)}) circle [radius=.05] node [below left]{$\tau'$};
\node at (3.55,0.75) {$Q_1$};
\node at (4.25,0.3) {$Q_2$};
\node at (5,-0.2) {$Q_3$};
\end{tikzpicture}
\caption{}\label{fig2}
\end{figure}

\begin{proof}
Note that 
\[\tilde u(\tau', \lambda + h') - \tilde u(\tau', \lambda + h) - \tilde u(\tau, \lambda + h') + \tilde u(\tau, \lambda + h) = \frac{1}{2} W(Q),\]
where $Q$ is the image of the rectangle $(\tau, \tau'] \times (\lambda+h, \lambda+h']$ under the rotation 
$(\tau, \lambda) \mapsto (\frac{\tau+\lambda}{\sqrt{2}}, \frac{-\tau+\lambda}{\sqrt{2}})$.
Suppose $h'-h \le \tau'-\tau$. Then $Q = Q_1 \cup Q_2 \cup Q_3$, where
\begin{align*}
Q_1 &= \Big\{ (t, x) : \frac{\tau+\lambda+h}{\sqrt{2}} < t \le \frac{\tau+\lambda+h'}{\sqrt{2}}, \sqrt{2}(\lambda+h)-s <x< -\sqrt{2}\tau+s \Big\},\\
Q_2 &= \Big\{ (t, x) : \frac{\tau+\lambda+h'}{\sqrt{2}} < t \le \frac{\tau'+\lambda+h}{\sqrt{2}}, \sqrt{2}(\lambda+h)-s <x\le \sqrt{2}(\lambda+h')-s \Big\},\\
Q_3 &= \Big\{ (t, x) : \frac{\tau'+\lambda+h}{\sqrt{2}} < t \le \frac{\tau'+\lambda+h'}{\sqrt{2}}, -\sqrt{2}\tau'+s \le x \le \sqrt{2}(\lambda+h')-s \Big\}.
\end{align*}
The sets $Q_1$, $Q_2$ and $Q_3$ are shown in Figure \ref{fig2}.
By \eqref{Eq:noise_cov2},
\begin{align*}
& \E \big[(\tilde u(\tau', \lambda + h') - \tilde u(\tau', \lambda + h) - \tilde u(\tau, \lambda + h') + \tilde u(\tau, \lambda + h))^2 \big]
 = \frac{1}{4}\, \E\big[W(Q)^2 \big]\\
& = \frac{1}{8\pi} \Bigg\{ \int_{\frac{\tau+\lambda+h}{\sqrt{2}}}^{\frac{\tau+\lambda+h'}{\sqrt{2}}} ds \int_{-\infty}^\infty
 \frac{C_\beta d\xi}{|\xi|^{1-\beta}} \big|\mathscr{F}{\bf 1}_{[\sqrt{2}(\lambda+h)-s, \, -\sqrt{2}\tau+s]}(\xi)\big|^2 \\
& \qquad + \int_{\frac{\tau+\lambda+h'}{\sqrt{2}}}^{\frac{\tau'+\lambda+h}{\sqrt{2}}} ds \int_{-\infty}^\infty \frac{C_\beta d\xi}
{|\xi|^{1-\beta}} \big|\mathscr{F}{\bf 1}_{[\sqrt{2}(\lambda+h)-s, \, \sqrt{2}(\lambda+h')-s]}(\xi) \big|^2\\
& \qquad + \int_{\frac{\tau'+\lambda+h}{\sqrt{2}}}^{\frac{\tau'+\lambda+h'}{\sqrt{2}}} ds \int_{-\infty}^\infty 
\frac{C_\beta d\xi}{|\xi|^{1-\beta}} \big|\mathscr{F}{\bf 1}_{[-\sqrt{2}\tau'+s, \, \sqrt{2}(\lambda+h')-s]}(\xi) \big|^2 \Bigg\}.
\end{align*}
Then, by Lemma \ref{Lem1},
\begin{align*}
& \E \big[(\tilde u(\tau', \lambda + h') - \tilde u(\tau', \lambda + h) - \tilde u(\tau, \lambda + h') + \tilde u(\tau, \lambda + h))^2 \big]\\
& = \frac{1}{2(2-\beta)(1-\beta)}\Bigg\{ \int_{\frac{\tau+\lambda+h}{\sqrt{2}}}^{\frac{\tau+\lambda+h'}{\sqrt{2}}} 
\big(2s - \sqrt{2}(\tau+\lambda+h)\big)^{2-\beta} ds\\
& \quad + \int_{\frac{\tau+\lambda+h'}{\sqrt{2}}}^{\frac{\tau'+\lambda+h}{\sqrt{2}}} \big(\sqrt{2}(h' - h) \big)^{2-\beta} ds
+ \int_{\frac{\tau'+\lambda+h}{\sqrt{2}}}^{\frac{\tau'+\lambda+h'}{\sqrt{2}}} \big( \sqrt{2}(\tau'+\lambda+h') - 2s \big)^{2-\beta} ds \Bigg\}\\
& = \frac{1}{2(2-\beta)(1-\beta)} \bigg\{ 2 \times \frac{2^{\frac{1-\beta}{2}}}{3-\beta} (h'-h)^{3-\beta} + 2^{\frac{1-\beta}{2}} 
(h' - h)^{2-\beta} \big[(\tau'-\tau) - (h' - h)\big] \bigg\}\\
& = \frac{1}{2}\, K_\beta^2 (h'-h)^{2-\beta} \Big\{ (\tau'-\tau) - \frac{1-\beta}{3-\beta}(h' - h)\big) \Big\}.
\end{align*}
This proves (\ref{Eq:rec-inc}) in the case of $h'-h \le \tau'-\tau$. The case of $h' - h > \tau' - \tau$ can be dealt with in a similar way. 
We omit the details.
\end{proof}

%Recall a standard result of large deviation (cf.~\cite{LS70, MS72}): If $\{Z(t) : t \in T \}$ is a continuous centered 
%Gaussian random field which is a.s.~bounded, then
%\begin{align}\label{large_dev0}
%\lim_{\gamma \to \infty} \frac{1}{\gamma^2} \log \P\left(\sup_{t \in T} Z(t) > \gamma\right) = 
%-\frac{1}{2\sup_{t \in T}\E(Z(t)^2)}.
%\end{align}
%By symmetry of the distribution of $\{Z(t) : t \in T \}$, we have
%\begin{align}\label{large_dev}
%\lim_{\gamma \to \infty} \frac{1}{\gamma^2} \log \P\left(\sup_{t \in T} |Z(t)| > \gamma\right) = 
%-\frac{1}{2\sup_{t \in T}\E(Z(t)^2)}.
%\end{align}

It is known that for establishing the law of the iterated logarithm for a Gaussian process at a fixed point, 
the following large deviation result  (cf.~\cite{LS70, MS72}) is useful:  If $\{Z(t), t \in T \}$ is a continuous 
centered Gaussian random field which is a.s.~bounded, then
\begin{align}\label{large_dev0}
\lim_{\gamma \to \infty} \frac{1}{\gamma^2} \log \P\left(\sup_{t \in T} Z(t) > \gamma\right) = -\frac{1}{2\sup_{t \in T}\E(Z(t)^2)}.
\end{align}
%By symmetry of the distribution of $\{Z(t) : t \in T \}$, we have
%\begin{align}\label{large_dev}
%\lim_{\gamma \to \infty} \frac{1}{\gamma^2} \log \P\left(\sup_{t \in T} |Z(t)| > \gamma\right) = 
%-\frac{1}{2\sup_{t \in T}\E(Z(t)^2)}.
%\end{align}
The result (\ref{large_dev0}), together with the scaling property of fractional Brownian sheets,  also allows 
Blath and Martin \cite{BM08}  to prove a simultaneous LIL and study the propagation of singularities of the 
semi-fractional Brownian sheet. However, it turns out that (\ref{large_dev0})  is not precise enough for proving 
Proposition \ref{Prop:sim_LIL_UB} due to the facts that the index set in the event $A_n$ in \eqref{Eq:An} depends 
on $n$ and $\tilde u$ does not have scaling property. To overcome this difficulty, we will make use of the 
following two results of Talagrand \cite[Theorem 2.4, Proposition 2.7]{T94} regarding tail probability estimates 
for general Gaussian processes in terms of metric entropy.

\begin{lemma}\label{Lem:T1}
Let $\{ Z(t), t \in T \}$ be a mean zero continuous Gaussian process and $\sigma_T^2 = \sup_{t \in T}\E[Z(t)^2]$.
Consider the canonical metric $d_Z$ on $T$ defined by  $d_Z(s, t) = \E[(Z(s) - Z(t))^2]^{1/2}$.
Assume that for some constant $M > \sigma_T$, some $\alpha > 0$ and 
some $0 < \eps_0 \le \sigma_T$,
\begin{equation*}
N(T, d_Z, \eps) \le \left( \frac{M}{\eps}\right)^\alpha \quad \text{for all } \eps < \eps_0,
\end{equation*}
where $N(T, d_Z, \eps)$ is the smallest number of $d_Z$-balls of radius $\varepsilon$ needed to cover $T$.
Then for any $\gamma > \sigma_T^2[(1+\sqrt{\alpha})/\eps_0]$, we have
\begin{equation}\label{Eq:T0}
\P\bigg\{ \sup_{t \in T} Z(t) \ge \gamma \bigg\} 
\le \left( \frac{KM\gamma}{\sqrt{\alpha}\,\sigma_T^2} \right)^\alpha
\Phi\left( \frac{\gamma}{\sigma_T} \right),
\end{equation}
where $\Phi(x) = (2\pi)^{-1/2} \int_x^\infty \exp(-z^2/2) \, dz$
and $K$ is a universal constant.
\end{lemma}

The conditions of Lemma \ref{Lem:T1} is easier to verify and the upper bound in (\ref{Eq:T0}) is more precise than 
that is given by \eqref{large_dev0} if $M/\sigma_T$ is not too large. However, the upper bound in (\ref{Eq:T0}) may 
not be useful when $M/\sigma_T$ becomes very large [this is the case for the even $A_n$ in \eqref{Eq:An} below.]
On the other hand, Lemma \ref{Lem:T} is more efficient if the variance of $ Z(t)$ attains its maximum 
at a unique point because the size of the set $T_\rho $ can be very small. 

\begin{lemma}\label{Lem:T}
Let $\{ Z(t), t \in T \}$ be a mean zero continuous Gaussian process as in Lemma \ref{Lem:T1}.
%Let $\sigma_T^2 = \sup_{t \in T}\E(|X(t)|^2)$, and, 
For $\rho > 0$, set
\[ T_\rho = \{ t \in T : \E[Z(t)^2] \ge \sigma_T^2 - \rho^2 \}. \]
%Define the canonical metric $d_Z$ on $T$ by 
%$d_Z(s, t) = (\E|Z(s) - Z(t)|^2)^{1/2}$.
Assume that there exist constants $v \ge w \ge 1$ such that for all $\rho > 0$, 
and $0 < \eps \le \rho(1+\sqrt{v})/\sqrt{w}$, we have
\[ N(T_\rho, d_Z, \eps) \le A \rho^w \eps^{-v}. \]
%where $N(T_\rho, d_Z, \eps)$ is the smallest number of $d_Z$-balls 
%of radius $\eps$ needed to cover $T_\rho$.
Then for any $\gamma > 2\sigma_T \sqrt{w}$, we have
\begin{equation*}
\P\bigg\{ \sup_{t \in T} Z(t) \ge \gamma \bigg\} 
\le \frac{Aw^{w/2}}{v^{v/2}} K^{v+w} \left( \frac{\gamma}{\sigma_T^2} \right)^{v-w}
\Phi\left(\frac{\gamma}{\sigma_T}\right).
\end{equation*}
%where $\Phi(x) = (2\pi)^{-1/2} \int_x^\infty \exp(-z^2/2) \, dz$
%and $K$ is a universal constant.
\end{lemma}

Now, we prove Proposition \ref{Prop:sim_LIL_UB}.

\begin{proof}[Proof of Proposition \ref{Prop:sim_LIL_UB}]
Fix $\lambda > 0$.
It suffices to show that for any $0 \le a < b < \infty$ and any $0 < \varepsilon < 1$,
\begin{equation}\label{Eq:sim_LIL_UB}
\P\left( \limsup_{h \to 0+} \frac{|\tilde{u}(\tau, \lambda + h) - \tilde{u}(\tau, \lambda)|}{\sqrt{(\tau+\lambda) h^{2-\beta} \log\log(1/h)}} 
\le (1+\varepsilon)K_\beta \textup{ for all } \tau \in [a, b]\right) = 1.
\end{equation}
Let $c \in [a, b]$, $\delta = (c+\lambda)\varepsilon/2$ and $d = c+\delta$. 
%Take $0 < \theta < 1$ such that $\theta(1+\varepsilon) > 1$.
Choose a real number $q$ such that $1 < q < (1+\varepsilon)^{1/(2-\beta)}$.
For every integer $n \ge 1$, consider the event
\begin{equation}\label{Eq:An}
A_n = \bigg\{ \sup_{\tau \in [c, d]} \sup_{h \in [0, q^{-n}]} \big|\tilde{u}(\tau, \lambda+h) - \tilde{u}(\tau, \lambda) \big| 
> \gamma_n \bigg\}, 
\end{equation}
where 
\[
\gamma_n = (1+\eps) K_\beta \sqrt{(c+\lambda)(q^{-n-1})^{2-\beta} \log\log q^n}.
\]
%Fix $n \ge 1$.
To estimate $\P(A_n) $, we will apply Lemma \ref{Lem:T}. Define $T = [c, d] \times [0, q^{-n}]$
and $Z(\tau, h) = \tilde u(\tau, \lambda+h) - \tilde u(\tau, \lambda)$
for $(\tau, h) \in T$. It follows from Lemma \ref{Lem2} that $\E\big[Z(\tau, h)^2\big]$ 
attains its unique maximum $\sigma^2_T$ at $(d, q^{-n})$, where
\[ \sigma^2_T = \frac 1 2 K_\beta^2 \big[(d+\lambda)q^{-n(2-\beta)} + (3-\beta)^{-1} q^{-n(3-\beta)}\big].\]

For any $(\tau, h), (\tau', h') \in T$, without loss of generality, we may assume that $\tau \le \tau'$. 
By the triangle inequality,
\begin{align*}
d_Z((\tau, h), (\tau', h'))
&\le \E\big[(Z(\tau, h) - Z(\tau, h'))^2\big]^{1/2} + \E\big[(Z(\tau', h') - Z(\tau, h'))^2\big]^{1/2}\\
& = \E\big[(\tilde u(\tau, \lambda+h) - \tilde u(\tau, \lambda+h'))^2\big]^{1/2}\\
& \quad + \E\big[(\tilde u(\tau', \lambda+h') - \tilde u(\tau', \lambda) - \tilde u(\tau, \lambda+h') + 
\tilde u(\tau, \lambda))^2\big]^{1/2}\\
& =: S_1 + S_2.
\end{align*}
By Lemma \ref{Lem2}, $S_1 \le C|h-h'|^{(2-\beta)/2}$.
The term $S_2$ can be evaluated  by using Lemma \ref{Lem3}.
In particular, if $|\tau-\tau'| \ge h'$, we have
\begin{align*}
S_2 \le \frac{1}{\sqrt 2} K_\beta\, q^{-n(2-\beta)/2} |\tau-\tau'|^{1/2}.
\end{align*}
If $|\tau-\tau'| < h'$, since $h' \le q^{-n}$ and $(2-\beta)/2 > 1/2$, we have
\begin{align*}
S_2 &\le \frac{1}{\sqrt 2} K_\beta |\tau-\tau'|^{(2-\beta)/2} (h')^{1/2}\\
& \le \frac{1}{\sqrt 2} K_\beta\, q^{-n(2-\beta)/2} |\tau-\tau'|^{1/2}.
\end{align*}
Hence,
\begin{equation}\label{d_Z}
d_Z((\tau, h), (\tau', h')) 
\le C(q^{-n(2-\beta)/2} |\tau-\tau'|^{1/2} + |h-h'|^{(2-\beta)/2}).
\end{equation}
Next, we estimate $N(T_\rho, d_Z, \eps)$, where
\[ T_\rho = \{ (\tau, h) \in T : \sigma^2_T - \E[Z(\tau, h)^2] \le \rho^2 \}.\]
%By Lemma \ref{Lem2},
%\[ \sigma^2_T = \frac 1 2 K_\beta^2 \big[(d+\lambda)q^{-n(2-\beta)} + (3-\beta)^{-1} q^{-n(3-\beta)}\big] \]
We can write $\sigma^2_T - \E[Z(\tau, h)^2]$ as the sum of three terms:
\begin{align*}
\frac 1 2 K_\beta^2 \Big[ (d-\tau) q^{-n(2-\beta)} + (\tau +\lambda) (q^{-n(2-\beta)} - h^{2-\beta}) + (3-\beta)^{-1}(q^{-n(3-\beta)} - h^{3-\beta}) \Big].
\end{align*}
It follows that $(\tau, h) \in T_\rho$ implies that each of the three terms 
is at most $\rho^2$. Then, together with the elementary inequality 
$x^{1/(2-\beta)}- y^{1/(2-\beta)} \le (x-y)^{1/(2-\beta)}$
for $x \ge y > 0$ and $0 < 1/(2-\beta) < 1$, we can deduce that
\begin{align*}
T_\rho \subset [d - C_1 q^{n(2-\beta)} \rho^2, d] \times [q^{-n} - C_2 \rho^{2/(2-\beta)}, q^{-n}]
\end{align*}
for some constants $C_1$ and $C_2$. This and \eqref{d_Z} imply that
\begin{align*}
N(T_\rho, d_Z, \eps) \le C_0(\rho/\eps)^{2+\frac{2}{2-\beta}}.
\end{align*}
Hence, we apply Lemma \ref{Lem:T} with $v = w = 2 + \frac{2}{2-\beta}$
and the standard estimate $\Phi(x) \le (2\pi)^{-1/2} e^{-x^2/2}$ for $x \ge 1$
to get that
\begin{align*}
\P(A_n) &\le C \exp\Big(-\frac{\gamma_n^2}{2\sigma^2_T}\Big) 
= (n \log q)^{-p_n},
\end{align*}
where
\begin{align*}
p_n = \frac{(1+\eps)^2}{q^{2-\beta}\big[\frac{d+\lambda}{c+\lambda} + (3-\beta)^{-1}(c+\lambda)^{-1} q^{-n}\big]}.
\end{align*}
Recall that $d =c+\delta$ and $\delta = (c+\lambda)\varepsilon/2$. 
When $n$ is large enough, 
$(3-\beta)^{-1}(c+\lambda)^{-1} q^{-n} \le \varepsilon/2$, 
which implies that
\[ p_n \ge \frac{1+\varepsilon}{q^{2-\beta}} > 1. \]
Hence $\sum_{n=1}^\infty \P(A_n) < \infty$ and 
by the Borel-Cantelli lemma, we have
$\P(A_n \textup{ i.o.}) = 0$.
%The same argument applied to the event
%\begin{align*} 
%B_n &= \bigg\{ \inf_{\tau \in [0, d]} \inf_{h \in [0, q^{-n}]} \Big(\tilde{u}(\tau, \lambda+h) - \tilde{u}(\tau, \lambda) \Big) < -\,\gamma_n \bigg\}\\
%& = \bigg\{ \sup_{\tau \in [0, d]} \sup_{h \in [0, q^{-n}]} \Big[- \Big(\tilde{u}(\tau, \lambda+h) - \tilde{u}(\tau, \lambda) \Big)\Big] > \gamma_n \bigg\}
%\end{align*}
%yields $\P(B_n \textup{ i.o.}) = 0$.
It follows that with probability 1,
\[ \sup_{\tau \in [c, d]} \sup_{h \in [q^{-n-1}, q^{-n}]} \frac{|\tilde{u}(\tau, \lambda+h) - \tilde{u}(\tau, \lambda)|}
{\sqrt{(c+\lambda)(q^{-n-1})^{2-\beta}\log\log q^n}} \le (1+\varepsilon) K_\beta 
\] 
eventually for all large $n$. Hence
\[ \P\left( \limsup_{h \to 0+}\frac{|\tilde{u}(\tau, \lambda+h) - \tilde{u}(\tau, \lambda)|}{\sqrt{(\tau+\lambda)h^{2-\beta} \log\log (1/h)}} 
\le (1+\varepsilon) K_\beta \textup{ for all } \tau \in [c, d]\right) = 1. \]
From this, we can deduce \eqref{Eq:sim_LIL_UB} by covering the interval $[a, b]$ by finitely many intervals $[c, d]$ of length $\delta$.
\end{proof}

\section{Simultaneous Law of Iterated Logarithm: Lower Bound}

Next, we prove the lower bound for the simultaneous LIL:

\begin{proposition}\label{Prop:sim_LIL_LB}
For any $\lambda > 0$,
\begin{equation}\label{Eq:sim_LIL_LB}
\P\left( \limsup_{h \to 0+} \frac{|\tilde{u}(\tau, \lambda+h) - \tilde{u}(\tau, \lambda)|}{\sqrt{(\tau+\lambda)h^{2-\beta} \log\log(1/h)}} 
\ge K_\beta \textup{ for all } \tau \in [0, \infty) \right) = 1,
\end{equation}
where $K_\beta$ is the constant in (\ref{Eq:K}) of Proposition \ref{Prop:sim_LIL_UB}. 
%i.e.
%\[K_\beta = \bigg( \frac{2^{(1-\beta)/2}}{{(2-\beta)(1-\beta)}}\bigg)^{1/2}.\]
\end{proposition}

\smallskip
Recall the following version of Borel-Cantelli lemma \cite[p.391]{R}.

\begin{lemma}\label{BClemma}
Let $\{A_n : n \ge 1\}$ be a sequence of events. If
\begin{enumerate}
\item[(i)] $\sum_{n=1}^\infty \P(A_n) = \infty$ and
\item[(ii)] $\displaystyle \liminf_{n \to \infty} \frac{\sum_{j=1}^n \sum_{k=1}^n \P(A_j \cap A_k)}{[\sum_{j=1}^n \P(A_j)]^2} = 1$,
\end{enumerate}
then $\P(A_n \textup{ i.o.}) = 1$.
\end{lemma}

We will also use the following lemma, which is essentially proved in \cite{S62}. For the sake of completeness, 
we provide a proof for this result.

\begin{lemma}\label{slepian}
Let $Z_1$ and $Z_2$ be jointly Gaussian random variables with $\E (Z_i) = 0$, $\E (Z_i^2) = 1$ and $\E (Z_1Z_2) = r$. 
Then for any $\gamma_1, \gamma_2 > 0$, there exists a number $r^*$ between $0$ and $r$ such that
\[ 
\P(Z_1 > \gamma_1, Z_2 > \gamma_2) - \P(Z_1 > \gamma_1)\P(Z_2 > \gamma_2) = r g(\gamma_1, \gamma_2; r^*),
 \]
where $g(x, y; r)$ is the standard bivariate Gaussian density with correlation $r$, i.e.
\[ g(x, y; r) = \frac{1}{2\pi(1-r^2)^{1/2}} \exp\left(-\frac{x^2 + y^2 - 2r xy}{2(1-r^2)}\right). \]
\end{lemma}

\begin{proof}
Let $\gamma_1, \gamma_2 > 0$ and $p(r) = \int_{\gamma_1}^\infty \int_{\gamma_2}^\infty g(x, y; r) \, dx\, dy$.
Define the Fourier transform of a function $f(x, y)$ as $\mathscr{F}f(\xi, \zeta) = \iint_{\mathbb{R}^2} e^{-i(x\xi + y\zeta)} f(x, y) \, dx\, dy$.
Note that
\[ g(x, y; r) = \frac{1}{(2\pi)^2} \iint_{\mathbb{R}^2} e^{i(x\xi + y \zeta)} [\mathscr{F}g(\ast\, ;r)](\xi, \zeta) \, d\xi \, d\zeta \]
and
\[ [\mathscr{F}g(\ast\, ;r)](\xi, \zeta) = e^{-\frac{1}{2}(\xi^2 +2r\xi \zeta + \zeta^2)}. \]
By the dominated convergence theorem,
\begin{align*}
\partial_r g(x, y; r) &= \frac{-1}{(2\pi)^2} \iint_{\mathbb{R}^2} e^{i(x\xi + y \zeta)} \xi \zeta [\mathscr{F}g(\ast\, ;r)](\xi, \zeta) \, d\xi \, d\zeta.
\end{align*}
Since $(i\xi)(i\zeta) \cdot \mathscr{F}f(\xi, \zeta) = [\mathscr{F}\partial_x \partial_y f](\xi, \zeta)$, we have
\begin{align*}
\partial_r g(x, y; r) = \frac{1}{(2\pi)^2} \iint_{\mathbb{R}^2} e^{i(x\xi + y \zeta)} [\mathscr{F}\partial_x \partial_y g(\ast; r)](\xi, \zeta) \, 
d\xi \, d\zeta = \partial_x \partial_y g(x, y; r).
\end{align*}
Therefore,
\[ \partial_r p = \int_{\gamma_1}^\infty \int_{\gamma_2}^\infty \partial_x \partial_y g(x, y; r)\, dx\, dy = g(\gamma_1, \gamma_2; r). \]
The mean value theorem implies that $p(r) - p(0) = r g(\gamma_1, \gamma_2; r^*)$ for some $r^*$ between $0$ and $r$, 
and hence the result.
\end{proof}

Let $\sigma$ and $\tilde{\sigma}$ be the canonical metric on $\R_+\times \R$ for $u$ and $\tilde{u}$ respectively, i.e.,
\[ {\sigma}[(t, x), (t', x')] = \E[({u}(t, x) - {u}(t',x'))^2]^{1/2}, \]
\[ \tilde{\sigma}[(\tau, \lambda), (\tau', \lambda')] = \E[(\tilde{u}(\tau, \lambda) - \tilde{u}(\tau',\lambda'))^2]^{1/2}. \]
For a rectangle $I = [a,a'] \times [-b, b]$, where $0 < a < a'< \infty$ and $0 < b < \infty$,
recall from  \cite[Proposition 4.1]{DS10} that there exist positive finite constants $C_1$ and $C_2$ such that
\begin{equation}\label{Eq:incre_SWE}
C_1 \big(|t-t'| + |x-x'|\big)^{(2-\beta)/2} \le \sigma[(t, x), (t',x')]
\le C_2 \big(|t-t'|+ |x-x'|\big)^{(2-\beta)/2}
\end{equation}
for all $(t, x), (t',x') \in I$.

The proof of the following lemma is based on the method in \cite{O84, O89}.

\begin{lemma}\label{Lem:large_dev}
Let $\tau > 0$, $\lambda > 0$ and $q > 1$. Then for all $0 < \varepsilon < 1$,
\begin{equation}
\P\left( \frac{\tilde{u}(\tau, \lambda + q^{-n}) - \tilde{u}(\tau, \lambda + q^{-n-1})}{\tilde\sigma[(\tau, \lambda + 
q^{-n}), (\tau, \lambda+q^{-n-1})]} \ge (1-\varepsilon)\sqrt{2\log\log q^n} \textup{ infinitely often in }n \right) = 1.
\end{equation}
\end{lemma}

\begin{proof}
For $n \ge 1$, let $A_n = \{Z_n > \gamma_n\}$, where
\[ Z_n = \frac{\tilde{u}(\tau, \lambda + q^{-n}) - \tilde{u}(\tau, \lambda+ q^{-n-1})}{\tilde\sigma[(\tau, \lambda 
+ q^{-n}), (\tau, \lambda + q^{-n-1})]} 
\]
and
\[ \gamma_n = (1-\varepsilon) \sqrt{2 \log\log q^n}. \]
We will complete the proof by showing that (i) and (ii) of Lemma \ref{BClemma} are satisfied. 
For (i), by using the standard estimate
\begin{align}\label{tail_est}
 \P(Z > x) \ge (2\sqrt{2\pi})^{-1} x^{-1} \exp(-x^2/2), \quad x > 1, 
\end{align}
for a standard Gaussian random variable $Z$, we derive that for large $n$,
\begin{align*}
\P(Z_n > \gamma_n) \ge \frac{C}{n^{(1-\varepsilon)^2}\sqrt{\log n}}
\end{align*}
and hence $\sum_{n=1}^\infty \P(A_n) = \infty$.

Next, we show that (ii) is satisfied. Since 
$$\sum_{j=1}^n \sum_{k=1}^n \big[\P(A_j \cap A_k) - \P(A_j)\P(A_k)\big] = 
\E\Bigg[\bigg(\sum_{j=1}^n \big({\bf 1}_{A_j} -  \P(A_j)\big)\bigg)^2\Bigg] \ge 0$$ 
and $\sum_{n=1}^\infty \P(A_n) = \infty$, it is enough to prove that
\begin{equation}\label{ii'}
\liminf_{n \to \infty} \frac{\sum_{1 \le j < k \le n} [\P(A_j \cap A_k) - \P(A_j) \P(A_k)]}{[\sum_{j=1}^n \P(A_j)]^2} \le 0.
\end{equation}
We are going to use Lemma \ref{slepian} to estimate the difference in the numerator.  First, we 
estimate the correlation $r_{jk}$ between $Z_j$ and $Z_k$ for $j < k$.
Let $(t_n, x_n) = (\frac{\tau+\lambda+q^{-n}}{\sqrt{2}}, \frac{-\tau+\lambda+q^{-n}}{\sqrt{2}})$.
By \eqref{Eq:noise_cov2}, we have
\begin{align*}
&\E \big[(\tilde{u}(\tau, \lambda+q^{-j}) - \tilde{u}(\tau, \lambda+q^{-j-1}))(\tilde{u}(\tau, \lambda + q^{-k}) - \tilde{u}(\tau, \lambda+q^{-k-1}))\big]\\
& = \frac{1}{4} \, \E\big[W\big(\Delta(t_j, x_j)\backslash\Delta(t_{j+1}, x_{j+1})\big)W\big(\Delta(t_k, x_k) \backslash \Delta(t_{k+1}, x_{k+1})\big)\big]\\
& = \frac{1}{8\pi} \int_0^\infty ds \int_{-\infty}^\infty \frac{C_\beta d\xi}{|\xi|^{1-\beta}} \mathscr{F}({\bf 1}_{\Delta(t_j, x_j)\backslash\Delta(t_{j+1}, x_{j+1})}(s, \cdot))(\xi) \overline{\mathscr{F}({\bf 1}_{\Delta(t_k, x_k) \backslash \Delta(t_{k+1}, x_{k+1})}(s, \cdot))(\xi)} \\
& = \frac{1}{8\pi} \int_0^{t_{k+1}} ds \int_{-\infty}^\infty \frac{C_\beta d\xi}{|\xi|^{1-\beta}} \mathscr{F}({\bf 1}_{[x_{k+1}+t_{k+1}-s, \,x_k+t_k-s]})(\xi) \overline{\mathscr{F}({\bf 1}_{[x_{j+1}+t_{j+1}-s, \, x_j+t_j-s]})(\xi)}\\
& \quad + \frac{1}{8\pi} \int_{t_{k+1}}^{t_k} ds \int_{-\infty}^\infty \frac{C_\beta d\xi}{|\xi|^{1-\beta}} \mathscr{F}({\bf 1}_{[x_k-t_k+s, \, x_k+t_k-s]})(\xi) \overline{\mathscr{F}({\bf 1}_{[x_{j+1}+t_{j+1}-s, \, x_j+t_j-s]})(\xi)}.
\end{align*}
Note that \eqref{Eq:noise_cov2} also implies that this covariance is nonnegative.
Then by Lemma \ref{Lem1},
\begin{align}\label{corr}
\notag &\E\big[(\tilde{u}(\tau, \lambda+q^{-j}) - \tilde{u}(\tau, \lambda+q^{-j-1}))(\tilde{u}(\tau, \lambda + q^{-k}) - \tilde{u}(\tau, \lambda+q^{-k-1}))\big]\\
\notag & = C t_{k+1} \left[ (q^{-j-1} - q^{-k})^{2-\beta} - (q^{-j-1} - q^{-k-1})^{2-\beta} + (q^{-j}-q^{-k-1})^{2-\beta} - (q^{-j} - q^{-k})^{2-\beta} \right]\\
\notag &\quad + C \int_{q^{-k-1}}^{q^{-k}}\left[ (q^{-j-1} - q^{-k})^{2-\beta} - (q^{-j-1} - s)^{2-\beta} + (q^{-j} - s)^{2-\beta} - (q^{-j} - q^{-k})^{2-\beta} \right] ds\\
&=: J_1 + J_2.
\end{align}
Let us consider the first term $J_1$.
By the mean value theorem, we can find some $a$ and $b$ such that
\[ q^{-j-1} - q^{-k} \le a \le q^{-j-1} - q^{-k-1} < q^{-j} - q^{-k} \le b \le q^{-j} - q^{-k-1}. \]
and
\begin{align*}
&(q^{-j-1} - q^{-k})^{2-\beta} - (q^{-j-1} - q^{-k-1})^{2-\beta} + (q^{-j}-q^{-k-1})^{2-\beta} - (q^{-j} - q^{-k})^{2-\beta}\\
&= (2-\beta)(b^{1-\beta} - a^{1-\beta})(q^{-k} - q^{-k-1})\\
& \le (2-\beta)[(q^{-j} - q^{-k-1})^{1-\beta} - (q^{-j-1} - q^{-k})^{1-\beta}]q^{-k}.
\end{align*}
Suppose $j \le k - 2$.
By the mean value theorem again, we can find some $\xi$ between $q^{-j} - q^{-k-1}$ and $q^{-j-1} - q^{-k}$ 
such that
\begin{align*}
(q^{-j} - q^{-k-1})^{1-\beta} - (q^{-j-1} - q^{-k})^{1-\beta} 
&= (1-\beta) \xi^{-\beta}[(q^{-j} - q^{-k-1}) - (q^{-j-1} - q^{-k})]\\
& \le (1- \beta) (q^{-j-1} - q^{-k})^{-\beta} q^{-j}\\
& \le (1- \beta) (q^{-j-1} - q^{-j-2})^{-\beta} q^{-j}\\
& \le (1-\beta) (q^{-1}-q^{-2}) (q^{-j})^{1-\beta}.
\end{align*}
It follows that
\begin{align*}
%&%(q^{-j-1} - q^{-k})^{2-\beta} - (q^{-j-1} - q^{-k-1})^{2-\beta} + (q^{-j}-q^{-k-1})^{2-\beta} - (q^{-j} - q^{-k})^{2-\beta}\\
J_1\le C (q^{-j})^{1-\beta}q^{-k}.
\end{align*}

Next, we consider the term $J_2$ in \eqref{corr}. For every $s \in [q^{-k-1}, q^{-k}]$, we can find 
some $\tilde{a}$ and $\tilde{b}$ (depending on $s$) such that
\[ q^{-j-1} - q^{-k} \le \tilde{a} \le q^{-j-1} - s < q^{-j} - q^{-k} \le \tilde{b} \le q^{-j} - s \]
and
\begin{align*}
& (q^{-j-1} - q^{-k})^{2-\beta} - (q^{-j-1} - s)^{2-\beta} + (q^{-j} - s)^{2-\beta} - (q^{-j} - q^{-k})^{2-\beta}\\
& = (2-\beta)(\tilde{b}^{1-\beta} - \tilde{a}^{1-\beta}) (q^{-k}-s)\\
& \le (2-\beta) [(q^{-j}-s)^{1-\beta} - (q^{-j-1} - q^{-k})^{1-\beta}] q^{-k}.
\end{align*}
For $j \le k-2$, by the mean value theorem again, there exists $\eta$ between $q^{-j}-s$ and $q^{-j-1} - q^{-k}$ such that
\begin{align*}
(q^{-j}-s)^{1-\beta} - (q^{-j-1} - q^{-k})^{1-\beta} 
&= (1-\beta) \eta^{-\beta} [(q^{-j}-s) - (q^{-j-1} - q^{-k})]\\
& \le (1-\beta) (q^{-j-1} - q^{-j-2})^{-\beta} q^{j}\\
& \le (1-\beta)(q^{-1} - q^{-2}) (q^{-j})^{1-\beta}.
\end{align*}
Then we have
\begin{align*}
J_2 &=C \int_{q^{-k-1}}^{q^{-k}} \left[(q^{-j-1} - q^{-k})^{2-\beta} - (q^{-j-1} - s)^{2-\beta} + (q^{-j} - s)^{2-\beta} 
- (q^{-j} - q^{-k})^{2-\beta} \right] ds\\
& \le C (q^{-j})^{1-\beta} (q^{-k})^2 \le C (q^{-j})^{1-\beta} q^{-k}.
\end{align*}
Therefore, by combining \eqref{corr}, the upper bounds for $J_1$, $J_2$, and recalling \eqref{Eq:incre_SWE}, we see that 
for $j \le k-2$, the correlation $r_{jk}$ between $Z_j$ and $Z_k$ satisfies
\begin{align}\label{r}
0 \le r_{jk} = \E (Z_j Z_k) \le \frac{C (q^{-j})^{1-\beta} q^{-k}}{(q^{-j})^{1-\beta/2}(q^{-k})^{1-\beta/2}} = C_0 (q^{-(k-j)})^{\beta/2} =: \xi_{jk}.
\end{align}

By \eqref{r}, we can choose a fixed $l \ge 2$ such that $r := \sup\{ r_{jk} : j \le k-l \} < 1$.
Since $\sum_{n=1}^\infty \P(A_n) = \infty$, in order to prove \eqref{ii'}, it suffices to 
prove that for any $\delta > 0$, there exists $m$ such that
\begin{align}\label{ii''}
\liminf_{n \to \infty} \frac{\sum_{k=m}^n \sum_{j=1}^{k-l}[\P(A_j \cap A_k) - \P(A_j)\P(A_k)]}{[\sum_{j=1}^n \P(A_j)]^2} \le \delta.
\end{align}
Let $\delta > 0$ be given and let $m$ be a large integer that will be chosen appropriately depending on $\delta$. 
Let $\rho_k = \frac{4}{(\beta/2)\log q} \log \gamma_k$, so that for $1 \le j \le k - \rho_k$,
\begin{equation}\label{Eq:eta}
\xi_{jk} \le C_0 \gamma_k^{-4}.
\end{equation}
Provided $m$ is large, $1 < k- \rho_k < k-l$ for all $k \ge m$.
By Lemma \ref{slepian}, we have
\begin{equation}\label{sum}
\sum_{k=m}^n \sum_{j=1}^{k-l}[\P(A_j \cap A_k) - \P(A_j)\P(A_k)] \le \left( \sum_{k=m}^n \sum_{j=1}^{\lfloor k-\rho_k \rfloor} 
+ \sum_{k=m}^n \sum_{j=\lfloor k-\rho_k\rfloor}^{k-l} \right) r_{jk} g(\gamma_j, \gamma_k; r_{jk}^*),
\end{equation}
where $r_{jk}^*$ is a number such that $0 \le r_{jk}^* \le r_{jk}$ for each $j, k$.
Let us consider the two sums on the right-hand side of \eqref{sum} separately.
By \eqref{r}, the first sum is
\begin{align*}
& \sum_{k=m}^n \sum_{j=1}^{\lfloor k-\rho_k \rfloor}\frac{r_{jk}}{2\pi({1-r_{jk}^{*2}})^{1/2}} \exp\left( - \frac{\gamma_j^2 + 
\gamma_k^2 - 2r_{jk}^* \gamma_j \gamma_k}{2(1-r_{jk}^{*2})} \right)\\
& \le \sum_{k=m}^n \sum_{j=1}^{\lfloor k-\rho_k \rfloor} \frac{\xi_{jk}\gamma_j \gamma_k}{2\pi(1-\xi_{jk}^2)^{1/2}} 
\exp\left( \frac{-r_{jk}^{*2}(\gamma_j^2+\gamma_k^2) + 2r_{jk}^* \gamma_j\gamma_k}{2(1-r_{jk}^{*2})}\right)
\gamma_j^{-1} e^{-\gamma_j^2/2} \gamma_k^{-1} e^{-\gamma_k^2/2}.
\end{align*}
Note that $\gamma_j < \gamma_k$ for $j < k$. Then by \eqref{tail_est}, \eqref{r} and \eqref{Eq:eta}, the sum is
\begin{align*}
\le 4 \sum_{k=m}^n \sum_{j=1}^{\lfloor k-\rho_k \rfloor} \frac{C_0\gamma_k^{-2}}{(1-C_0^2\gamma_k^{-8})^{1/2}} 
\exp\left( \frac{C_0\gamma_k^{-2}}{1-C_0^2\gamma_k^{-8}} \right) \P(A_j)\P(A_k).
\end{align*}
Since $\gamma_k \to \infty$, we may choose $m$ to be large enough such that this sum is $\le \delta [\sum_{j=1}^n \P(A_j)]^2$.

By \eqref{tail_est}, the second sum on the right-hand side of \eqref{sum} is
\begin{align*}
& \sum_{k=m}^n \sum_{j=\lfloor k-\rho_k \rfloor}^{k-l} \frac{r_{jk}}{2\pi({1-r_{jk}^{*2}})^{1/2}} \exp\left( - \frac{\gamma_j^2 
+ \gamma_k^2 - 2r_{jk}^* \gamma_j \gamma_k}{2(1-r_{jk}^{*2})} \right)\\
& \le \sum_{k=m}^n \sum_{j=\lfloor k-\rho_k \rfloor}^{k-l} \frac{r_{jk}\gamma_j}{2\pi(1-r_{jk}^{*2})^{1/2}}
\exp\left( -\frac{(\gamma_k - r_{jk}^* \gamma_j)^2}{2(1-r_{jk}^{*2})} \right) \gamma_j^{-1} e^{-\gamma_j^2/2}\\
& \le \frac{2}{\sqrt{2\pi}} \sum_{k=m}^n \sum_{j=\lfloor k-\rho_k \rfloor}^{k-l} \frac{\gamma_k}{(1-r_{jk}^{*2})^{1/2}} 
\exp\left(-\frac{(1-r_{jk}^*)^2\gamma_k^2}{2(1-r_{jk}^{*2})} \right) \P(A_j).
\end{align*}
Recall that $r = \sup\{r_{jk} : j \le k-l\} < 1$. Moreover, if $m$ is large enough, then
\[ \frac{\gamma_k \log \gamma_k}{(1-r^2)^{1/2}} \exp\left(-\frac{(1-r)\gamma_k^2}{2(1+r)} \right) \le \delta \]
and $k-\rho_k > k/2$ for all $k \ge m$, so that the last sum above is
\begin{align*}
& \le \frac{2}{\sqrt{2\pi}}\sum_{k=m}^n  \frac{\rho_k\gamma_k}{(1-r^2)^{1/2}} \exp\left(-\frac{(1-r)\gamma_k^2}{2(1+r)}  \right) 
\P(A_{\lfloor k-\rho_k \rfloor})\\
& \le C \sum_{k=m}^n  \frac{\gamma_k \log \gamma_k}{(1-r^2)^{1/2}} \exp\left(-\frac{(1-r)\gamma_k^2}{2(1+r)} \right) 
\P(A_{\lfloor k/2 \rfloor})\\
& \le 2C\delta \sum_{k=1}^n \P(A_k).
\end{align*}
We get that
\[ \sum_{k=m}^n \sum_{j=1}^{k-l} [\P(A_j \cap A_k) - \P(A_j)\P(A_k)] \le \delta \Bigg( \sum_{j=1}^n \P(A_j) \Bigg)^2 
+ 2C\delta \sum_{j=1}^n \P(A_j). \]
Hence \eqref{ii''} follows and the proof of Lemma \ref{Lem:large_dev} is complete.
\end{proof}

We now come to the proof of Proposition \ref{Prop:sim_LIL_LB}.

\begin{proof}[Proof of Proposition \ref{Prop:sim_LIL_LB}]
Fix $\lambda > 0$. It suffices to show that for any $0 \le a < b < \infty$ and $0 < \varepsilon < 1$,
\begin{equation}\label{Eq:LB}
\P\left( \limsup_{h \to 0+} \frac{|\tilde{u}(\tau, \lambda+h) - \tilde{u}(\tau, \lambda)|}{\sqrt{(\tau+\lambda)h^{2-\beta} \log\log(1/h)}} 
\ge (1-\varepsilon) K_\beta \textup{ for all } \tau \in [a, b] \right) = 1.
\end{equation}
To this end, let us fix $a, b$ and $\varepsilon$ for the rest of the proof.

Note that when $q$ is large, $q^{-\frac{2-\beta}{2}} (1 + \frac{q^{-n-1}}{\tau+\lambda})^{1/2} < \varepsilon/4$ uniformly for all $\tau \in [a, b]$.
So we can choose and fix a large $q > 1$ such that
\begin{equation}\label{Eq:q}
(1-\varepsilon/4) \bigg(\frac{q-1}{q} \bigg)^{\frac{2-\beta}{2}} - q^{-\frac{2-\beta}{2}} \bigg(1+\frac{q^{-n-1}}{\tau+\lambda}\bigg)^{1/2} - (1-\varepsilon) > \varepsilon/4
\end{equation}
for all $\tau \in [a, b]$.
We also choose $\delta > 0$ small such that
\begin{equation}\label{Eq:delta}
\frac{\lambda (\varepsilon/4)^2}{\delta} > 1.
\end{equation}
Since we can cover $[a, b]$ by finitely many intervals $[c, d]$ of length $\delta$, we only need to show that
\eqref{Eq:LB} holds for all $\tau \in [c, d]$, where $[c, d] \subset [a, b]$ and $d = c+\delta$.

Let us define the increment of $\tilde{u}$ over a rectangle $(\tau, \tau'] \times (\lambda, \lambda']$ by
\[ \Delta \tilde{u}((\tau, \tau'] \times (\lambda, \lambda']) = \tilde{u}(\tau', \lambda') - \tilde{u}(\tau, \lambda') - \tilde{u}(\tau', \lambda) + \tilde{u}(\tau, \lambda).  \]
Then for all $\tau \in [c, d]$ we can write
\begin{align}\label{Eq:u_decomp}
\begin{aligned}
\tilde{u}(\tau, \lambda + q^{-n}) - \tilde{u}(\tau, \lambda)
&= \tilde{u}(d, \lambda + q^{-n}) - \tilde{u}(d, \lambda + q^{-n-1})\\
& \quad + \tilde{u}(\tau, \lambda + q^{-n-1}) - \tilde{u}(\tau, \lambda)\\
& \quad - \Delta \tilde{u}((\tau, d] \times (\lambda + q^{-n-1}, \lambda + q^{-n}]).
\end{aligned}
\end{align}
By Lemma \ref{Lem:large_dev}, we have
\[ \frac{|\tilde{u}(d, \lambda + q^{-n}) - \tilde{u}(d, \lambda + q^{-n-1})|}{\tilde{\sigma}[(d, \lambda + q^{-n}), (d, \lambda + q^{-n-1})]} 
\ge (1-\varepsilon/4) \sqrt{2 \log\log q^n} \]
infinitely often in $n$ with probability 1. By Lemma \ref{Lem2},
\begin{align*}
&\tilde\sigma[(d, \lambda + q^{-n}), (d, \lambda + q^{-n-1})]\\
& = \frac{K_\beta}{\sqrt{2}} \sqrt{(d+\lambda + q^{-n-1})(q^{-n}-q^{-n-1})^{2-\beta} + (3-\beta)^{-1} (q^{-n} - q^{-n-1})^{3-\beta}},
\end{align*}
so we have
\begin{equation}\label{Eq:LB_u_tilde}
|\tilde{u}(d, \lambda + q^{-n}) - \tilde{u}(d, \lambda + q^{-n-1})| \ge (1-\varepsilon/4) K_\beta\sqrt{(d + \lambda)(q^{-n} - q^{-n-1})^{2-\beta}\log\log q^n} 
\end{equation}
infinitely often in $n$ with probability 1.
Also, by Proposition \ref{Prop:sim_LIL_UB}, with probability 1, for all $\tau \in [c, d]$ simultaneously,
\begin{equation}\label{Eq:UB_u_tilde}
|\tilde{u}(\tau, \lambda + q^{-n-1}) - \tilde{u}(\tau, \lambda)| \le K_\beta \sqrt{(\tau + \lambda + q^{-n-1}) (q^{-n-1})^{2-\beta}\log\log q^{n}}
\end{equation}
eventually for all large $n$.

Next, we derive a bound for the term $\Delta \tilde{u}((\tau, d] \times (\lambda + q^{-n-1}, \lambda + q^{-n}])$.
For $\tau \in [c, d]$, let
\[ \phi_n(\tau) = (1-\varepsilon/4) \bigg( \frac{q-1}{q} \bigg)^{\frac{2-\beta}{2}} (d + \lambda)^{1/2} - q^{-\frac{2-\beta}{2}} 
(\tau+\lambda+q^{-n-1})^{1/2} - (1-\varepsilon) (\tau + \lambda)^{1/2}. \]
Consider the event
\[ A_n = \left\{ \sup_{\tau \in [c, d]} |\Delta\tilde{u}((\tau, d] \times (\lambda+q^{-n-1}, \lambda+q^{-n}])| > \gamma_n \right\}, \]
where 
\[\gamma_n = K_\beta \,\phi_n(d) \sqrt{(q^{-n})^{2-\beta}\log\log q^n}.\]
Consider $n$ large enough such that $q^{-n}-q^{-n-1} \le d-c$.
Then
\[ \P(A_n) \le \P(A_n^1) + \P(A_n^2), \]
where
\begin{align*}
A_n^1 &= \left\{ \sup_{\tau \in [c, d - (q^{-n}-q^{-n-1})]} |\Delta\tilde{u}((\tau, d] \times (\lambda+q^{-n-1}, \lambda+q^{-n}])| > \gamma_n \right\},\\
A_n^2 &= \left\{ \sup_{\tau \in [d-(q^{-n}-q^{-n-1}), d]} |\Delta\tilde{u}((\tau, d] \times (\lambda+q^{-n-1}, \lambda+q^{-n}])| > \gamma_n \right\}.
\end{align*}
We first estimate $\P(A_n^1)$ using Lemma \ref{Lem:T}.
Let $Z(\tau) = \Delta \tilde u((\tau, d] \times (\lambda+q^{-n-1}, \lambda+q^{-n}])$,
$T = [c, d - (q^{-n}-q^{-n-1})]$ and 
\begin{align*}
d_Z(\tau, \tau') &:= \E[(Z(\tau) - Z(\tau'))^2]^{1/2}\\
& = \E[(\tilde u(\tau', \lambda+q^{-n}) - \tilde u(\tau', \lambda+q^{-n-1}) - \tilde u(\tau, \lambda+q^{-n}) + \tilde u(\tau, \lambda+q^{-n-1}))^2]^{1/2}.
\end{align*}
By Lemma \ref{Lem3}, 
\begin{align}\label{dZ}
\notag
d_Z(\tau, \tau')
&\le \begin{cases}
C(q^{-n} - q^{-n-1})^{(2-\beta)/2} |\tau-\tau'|^{1/2} & \text{if } q^{-n} - q^{-n-1} \le |\tau - \tau'|\\
C|\tau-\tau'|^{(2-\beta)/2} (q^{-n} - q^{-n-1})^{1/2} & \text{if } q^{-n} - q^{-n-1} > |\tau-\tau'|
\end{cases}\\
& \le C q^{-n(2-\beta)/2} |\tau-\tau'|^{1/2}.
\end{align}
Let $\sigma^2_T = \sup_{\tau \in T} \E[Z(\tau)^2]$ and
$T_\rho = \{ \tau \in T : \sigma^2_T -\E[Z(\tau)^2] \le \rho^2 \}$.
For all $\tau \in T$, $d-\tau \ge q^{-n}-q^{-n-1}$, so by Lemma \ref{Lem3},
\[ \sigma^2_T
= \frac 1 2 K^2_\beta (q^{-n}-q^{-n-1})^{2-\beta}\Big[(d-c) - \frac{1-\beta}{3-\beta}(q^{-n}-q^{-n-1})\Big]\]
and
\[ \sigma^2_T - \E[Z(\tau)^2]
= \frac 1 2 K^2_\beta (q^{-n}-q^{-n-1})^{2-\beta}(d-\tau).\]
Hence, $T_\rho \subset [d - C q^{n(2-\beta)} \rho^2, d]$.
This and \eqref{dZ} imply that $N(T_\rho, d_Z, \eps) \le C(\rho/\eps)^2$.
Now, we can apply Lemma \ref{Lem:T} with $v = w = 2$ to get that for $n$ large,
\[ \P(A_n^1) \le C \exp\left( - \frac{\gamma_n^2}{2\sigma_T^2} \right)
\le (n \log q)^{-p_n}, \]
where
\[ p_n = \frac{1}{d-c} \bigg( \frac{q}{q-1} \bigg)^{2-\beta} \phi_n(d)^2. \]
Recall that $d= c+ \delta$. By \eqref{Eq:q},
$\phi_n(d)^2 > (d+\lambda)(\eps/4)^2$ for all $n$, so
$p_n > \frac{\lambda (\varepsilon/4)^2}{\delta} > 1$ by \eqref{Eq:delta}.
Hence, $\sum_{n=1}^\infty \P(A_n^1) < \infty$.

To get an upper bound for  $\P(A_n^2)$, we take $T' = [d- (q^{-n}-q^{-n-1}), d]$. Since the size of $T'$ is 
quite small, we can apply Lemma \ref{Lem:T1}. Notice that  $\sigma_{T'}^2 \le C q^{-n(3-\beta)}$,
and by \eqref{dZ},
\[N(T', d_Z, \eps) \le C \frac{q^{-n}-q^{-n-1}}{q^{n(2-\beta)}\eps^2} \le C q^{-n(3-\beta)} \eps^{-2}.\]
Therefore, by Lemma \ref{Lem:T1}, for $n$ large, we get that
\[ \P(A_n^2) \le C \phi_n(d)^2(q^n \log n) \exp(-C' \phi_n(d)^2 q^n \log n).\]
Since $(d+\lambda)(\eps/4)^2\le \phi_n(d)^2 \le d+\lambda$ for all $n$, 
we have $\sum_{n=1}^\infty \P(A_n^2) < \infty$.

Hence, $\P(A_n \textup{ i.o.}) = 0$ by the Borel--Cantelli lemma.
Then, by the monotonic decreasing property of $\phi$, this implies that, 
with probability 1, simultaneously for all $\tau \in [c, d]$,
\begin{equation}\label{Eq:UB_Delta_u}
\big|\Delta \tilde{u}((\tau, d] \times (\lambda + q^{-n-1}, \lambda+q^{-n}]) \big| \le 
K_\beta \phi_n(\tau) \sqrt{(q^{-n})^{2-\beta} \log\log q^n}
\end{equation}
eventually for all large $n$.
By \eqref{Eq:u_decomp} and the triangle inequality,
\begin{align*}
\big|\tilde{u}(\tau, \lambda + q^{-n}) - \tilde{u}(\tau, \lambda)\big|
&\ge \big|\tilde{u}(d, \lambda + q^{-n}) - \tilde{u}(d, \lambda + q^{-n-1}) \big|\\
& \qquad  - \big|\tilde{u}(\tau, \lambda + q^{-n-1}) - \tilde{u}(\tau, \lambda)\big|\\
&  \qquad  - \big|\Delta \tilde{u}((\tau, d] \times (\lambda + q^{-n-1}, \lambda + q^{-n}])\big|.
\end{align*}
Then \eqref{Eq:LB_u_tilde}, \eqref{Eq:UB_u_tilde} and \eqref{Eq:UB_Delta_u} together imply that with probability 1, 
for all $\tau \in [c, d]$ simultaneously,
\begin{align*}
&\big|\tilde{u}(\tau, \lambda+q^{-n}) - \tilde{u}(\tau, \lambda) \big|\\
& \ge \bigg[ (1-\varepsilon/4)\bigg(\frac{q-1}{q}\bigg)^{\frac{2-\beta}{2}}(d+\lambda)^{1/2} - q^{-\frac{2-\beta}{2}}
 (\tau+\lambda + q^{-n-1})^{1/2} - \phi_n(\tau)\bigg] \\
& \qquad \qquad \qquad \times K_\beta \sqrt{(q^{-n})^{2-\beta} \log\log q^n}\\
& \ge (1-\varepsilon) K_\beta \sqrt{(\tau+\lambda) (q^{-n})^{2-\beta} \log\log q^n}
\end{align*}
infinitely often in $n$.
This shows that \eqref{Eq:LB} holds for all $\tau \in [c, d]$ and concludes the proof of Proposition \ref{Prop:sim_LIL_LB}.
\end{proof}

\section{Singularities and Their Propagation}

In this section, we study the existence and propagation of singularities of the stochastic wave equation \eqref{SWE}.
The main result is Theorem \ref{Thm:sing}.

Let us first discuss the interpretation of singularities and how they may arise.
Propositions \ref{Prop:sim_LIL_UB} and \ref{Prop:sim_LIL_LB} imply that LIL holds at any fixed point $(t, x)$:
\[  \limsup_{h \to 0+} \frac{|u(t+\frac{h}{\sqrt{2}}, x+\frac{h}{\sqrt{2}}) - u(t, x)|}{\sqrt{h^{2-\beta} \log\log(1/h)}} 
= K_\beta(\sqrt{2}t)^{1/2} \quad \text{a.s.} \]
It describes the rate of local oscillation of $u$ when $(t, x)$ is fixed. On the other hand, when $(t, x)$ is not fixed, 
the behavior will be different. Indeed, from the uniform modulus of continuity in Theorem 3.1 of \cite{LX19}, 
we know that for $I = [a,a'] \times [-b, b]$,  where $0 < a < a'$ and $b > 0$, there exists a positive finite constant
 $K$ such that
\[ \lim_{h \to 0+} \sup_{\substack{(t, x), (t',x') \in I:\\ 0 < \sigma[(t, x), (t',x')] \le h}}
\frac{|u(t', x') - u(t,x)|}{\sigma[(t, x),(t',x')]\sqrt{\log(1+\sigma[(t,x),(t',x')]^{-1})}} = K
\quad \text{a.s.} \]
Recalling \eqref{Eq:incre_SWE}, this result shows that the largest oscillation in $I$ is of order 
$\sqrt{h^{2-\beta}\log(1/h)}$, which is larger than $\sqrt{h^{2-\beta}\log\log(1/h)}$ as specified by the LIL.
It suggests that the LIL does not hold simultaneously for all
$(t, x) \in I$ and there may be (random) exceptional points with much larger oscillation.
Therefore, we can define singularities as such points where the LIL fails.
More precisely, we say that $(\tau, \lambda)$ is a \emph{singular point} of $\tilde{u}$ in the $\lambda$-direction if
\[ \limsup_{h \to 0+} \frac{|\tilde{u}(\tau, \lambda+h) - \tilde{u}(\tau, \lambda)|}{\sqrt{h^{2-\beta} \log\log(1/h)}} = \infty \]
and a \emph{singular point} in the $\tau$-direction if
\[ \limsup_{h \to 0+} \frac{|\tilde{u}(\tau + h, \lambda) - \tilde{u}(\tau, \lambda)|}{\sqrt{h^{2-\beta} \log\log(1/h)}} = \infty. \]
Our goal is to justify the existence of random singular points and study their propagation.

Fix $\tau_0 > 0$. Let us decompose $\tilde{u}$ into $\tilde{u}_1 + \tilde{u}_2$, where
\[\tilde{u}_i(\tau, \lambda) = u_i\Big(\frac{\tau+\lambda}{\sqrt{2}}, \frac{-\tau+\lambda}{\sqrt{2}}\Big), \quad i = 1,2,\]
and
\begin{align*}
{u}_1(t, x) &= \frac{1}{2}\, W\Big(\Delta(t, x) \cap \Pi(\tau_0) \Big),\\
{u}_2(t, x) &= \frac{1}{2}\, W\Big(\Delta(t, x) \cap \Pi(\tau_0)^c\Big),
\end{align*}
where $\Pi(\tau_0) $ is the vertical stripe $\Pi(\tau_0) = \big\{ (s, y): 0 \le s < \tau_0/\sqrt{2}, y \in \R \big\}$ 
and $\Pi(\tau_0)^c$ is the complement of $\Pi(\tau_0)$.

Let $\mathscr{F}_{\tau_0}$ be the $\sigma$-field generated by $\{W\big(B \cap \Pi(\tau_0)\big) : 
B \in \mathscr{B}_b(\R^2)\}$ and the $\P$-null sets.
Note that $\mathscr{F}_{\tau_0}$ is independent of the process $\tilde{u}_2$.

Following the approach of Walsh \cite{W82} and Blath and Martin \cite{BM08}, we will use Meyer's section 
theorem to prove the existence of a random singularity. Let us recall Meyer's section theorem (\cite{D}, 
Theorem 37, p.18):

Let $(\Omega, \mathscr{G}, \P)$ be a complete probability space and $S$ be a
$\mathscr{B}(\mathbb{R}_+) \times \mathscr{G}$-measurable subset of
$\mathbb{R}_+ \times \Omega$. Then there exists a $\mathscr{G}$-measurable
random variable $T$ with values in $(0, \infty]$ such that
\begin{enumerate}
\item[(a)] the graph of T, denoted by
$[T] := \{ (t, \omega) \in \mathbb{R}_+ \times \Omega : T(\omega) = t \}$,
is contained in $S$;
\item[(b)] $\{ T < \infty \}$ is equal to the projection $\pi(S)$ of $S$ onto $\Omega$.
\end{enumerate}

\smallskip

\begin{lemma}\label{Lem:u_1}
Let $\tau_0 > 0$. Then there exists a positive, finite, $\mathscr{F}_{\tau_0}$-measurable random variable $\Lambda$ 
such that
\[ \limsup_{h \to 0+} \frac{|\tilde{u}_1(\tau_0, \Lambda + h) - \tilde{u}_1(\tau_0, \Lambda)|}{\sqrt{h^{2-\beta}\log\log(1/h)}} 
= \infty \quad \text{a.s.} \]
\end{lemma}

\begin{proof}
Note that
\[\limsup_{h \to 0+} \frac{|\tilde{u}_1(\tau_0, \Lambda + h) - \tilde{u}_1(\tau_0, \Lambda)|}{\sqrt{h^{2-\beta}\log\log(1/h)}} 
= \limsup_{h \to 0+} \frac{|\tilde{v}_1(\tau_0, \Lambda + h) - \tilde{v}_1(\tau_0, \Lambda)|}{\sqrt{h^{2-\beta}\log\log(1/h)}}, 
\]
where $\tilde{v}_1(\tau_0, \lambda) = \tilde{u}_1(\tau_0, \lambda) - \tilde{u}_1(\tau_0, 0)$.
The covariance for the process $\{ \tilde{v}_1(\tau_0, \lambda), \lambda \ge 0 \}$ is
\begin{align*}
\E[\tilde{v}_1(\tau_0, \lambda) \tilde{v}_1(\tau_0, \lambda')] = \frac{1}{4}\, \E[W(A_\lambda)W(A_{\lambda'})]
\end{align*}
for $\lambda, \lambda' \ge 0$, where $A_\lambda = \{ (t, x) : 0 \le t < \tau_0/\sqrt{2}, -t < x \le \sqrt{2}\lambda - t \}$.
By \eqref{Eq:noise_cov2} and Lemma \ref{Lem1},
\begin{align*}
\E[\tilde{v}_1(\tau_0, \lambda) \tilde{v}_1(\tau_0, \lambda')] 
& = \frac{1}{8\pi} \int_0^{\tau_0/\sqrt{2}} ds \int_{-\infty}^\infty \frac{C_\beta d\xi}{|\xi|^{1-\beta}}
 \mathscr{F}{\bf 1}_{[-s, \, \sqrt{2}\lambda - s]}(\xi) \overline{\mathscr{F}{\bf 1}_{[-s, \, \sqrt{2}\lambda' - s]}(\xi)}\\
& = \frac{1}{4(2-\beta)(1-\beta)} \int_0^{\tau_0/\sqrt{2}} \Big( |\sqrt{2}\lambda|^{2-\beta} + |\sqrt{2}\lambda'|^{2-\beta} 
- |\sqrt{2}\lambda - \sqrt{2}\lambda'|^{2-\beta} \Big) ds\\
& = \frac{2^{-(3+\beta)/2}\,\tau_0}{(2-\beta)(1-\beta)} \Big( |\lambda|^{2-\beta} + |\lambda'|^{2-\beta} - |\lambda - \lambda'|^{2-\beta} \Big).
\end{align*}
It follows that $\{ C_0 \tilde{v}_1(\tau_0, \lambda), \lambda \ge 0 \}$ is a fractional Brownian motion of Hurst parameter 
$(2-\beta)/2$ for some constant $C_0$ depending on $\tau_0$ and $\beta$.

Let
\[ 
S = \Bigg\{ (\lambda , \omega) \in \R_+ \times \Omega : \limsup_{h \to 0+} \frac{|\tilde{v}_1(\tau_0, \lambda + h)(\omega) 
- \tilde{v}_1(\tau_0, \lambda)(\omega)|}{\sqrt{h^{2-\beta} \log\log(1/h)}} = \infty \Bigg\}. \]
Then $S$ is $\mathscr{B}(\R_+) \times \mathscr{F}_{\tau_0}$-measurable.
Using Meyer's section theorem, we can find a positive $\mathscr{F}_{\tau_0}$-measurable random variable $\Lambda$ such that
(a) $[\Lambda] \subset S$, and (b) $\pi(S) = \{\Lambda < \infty\}$.

We claim that $\Lambda < \infty$ a.s. Indeed, by the uniform modulus of continuity for fractional Brownian motion 
(cf. \cite{KS00}, Theorem 1.1), for any $0 \le a < b$,
\begin{equation}\label{fBm:mod_cont}
\limsup_{h \to 0+} \sup_{\lambda\in [a, b]} \frac{|\tilde{v}_1(\tau_0, \lambda+h) - \tilde{v}_1(\tau_0, \lambda)|}{\sqrt{h^{2-\beta} \log(1/h)}} 
= C_0^{-1} \sqrt{2} \quad \text{a.s.}
\end{equation}
We now use an argument with nested intervals (cf.~\cite{OT74}, Theorem 1) to show the existence of a random 
$\lambda^*$ such that 
\begin{equation}\label{tilde_v_singularity}
\limsup_{h \to 0+} \frac{|\tilde{v}_1(\tau_0, \lambda^*+h) - \tilde{v}_1(\tau_0, \lambda^*)|}{\sqrt{h^{2-\beta} \log\log(1/h)}} = \infty
\end{equation}
with probability 1.
First take an event $\Omega^*$ of probability 1 such that \eqref{fBm:mod_cont} holds for all intervals $[a, b]$, 
where $a$ and $b$ are rational numbers. Let 
\[\varphi(h) = \frac{1}{2}C_0^{-1}\sqrt{2h^{2-\beta} \log(1/h)}.\]
Let $h_0 > 0$ be small such that $\varphi$ is increasing on $[0, h_0]$. 
For an $\omega \in \Omega^*$, we define two sequences $(\lambda_n), (\lambda_n')$ as follows.
By \eqref{fBm:mod_cont}, we can choose $\lambda_1, \lambda_1'$, say in $[1, 2]$, with $\lambda_1 < \lambda_1'$ 
such that $\lambda_1'-\lambda_1 < h_0$ and 
\[|\tilde{v}_1(\tau_0, \lambda_1') - \tilde{v}_1(\tau_0, \lambda_1)| > \varphi(\lambda_1'-\lambda_1).\]
Suppose $n\ge 1$ and $\lambda_n$, $\lambda_n'$ are chosen with $0 < \lambda'_n - \lambda_n \le 2^{-n+1}$ and
\[|\tilde{v}_1(\tau_0, \lambda_n') - \tilde{v}_1(\tau_0, \lambda_n)| > \varphi(\lambda_n'-\lambda_n).\]
Since $\tilde{v}_1$ is continuous in $\lambda$ and $\varphi(h)$ is increasing for $h$ small, we can find some 
$\tilde{\lambda}_n$ such that $\lambda_n < \tilde{\lambda}_n < \min\{ \lambda_n', \lambda_n + 2^{-n} \}$ and
\begin{equation}\label{Eq:lower_function}
|\tilde{v}_1(\tau_0, \lambda_n') - \tilde{v}_1(\tau_0, \lambda)| > \varphi(\lambda_n' - \lambda) \quad 
\text{ for all }\lambda \in [\lambda_n, \tilde{\lambda}_n].
\end{equation}
Then we can apply \eqref{fBm:mod_cont} for a rational interval $[a, b] \subseteq [\lambda_n, \tilde{\lambda}_n]$ to find 
$\lambda_{n+1}$ and $\lambda_{n+1}'$ such that $\lambda_n \le \lambda_{n+1} < \lambda_{n+1}' \le \tilde{\lambda}_n$ and
\[ |\tilde{v}_1(\tau_0, \lambda_{n+1}') - \tilde{v}_1(\tau_0, \lambda_{n+1})| > \varphi(\lambda_{n+1}' - \lambda_{n+1}). \]
We obtain a sequence of nested intervals $[\lambda_1, \lambda_1'] \supset [\lambda_2, \lambda_2'] \supset \cdots$ 
with lengths $\lambda_n' - \lambda_n \le 2^{-n+1}$. Therefore, the intervals contain a common point $\lambda^* \in [1, 2]$ 
such that $\lambda_{n+1}' \downarrow \lambda^*$. Since 
$\lambda^* \in [\lambda_n, \tilde{\lambda}_n]$ for all $n$, by \eqref{Eq:lower_function} we have
\[ |\tilde{v}_1(\tau_0, \lambda_n') - \tilde{v}_1(\tau_0, \lambda^*)| > \varphi(\lambda_n' - \lambda^*). \]
Hence, for each $\omega \in \Omega^*$, there is at least one $\lambda^* > 0$ (depending on $\omega$)
such that \eqref{tilde_v_singularity} holds.
It implies that $\Omega^* \subset \pi(S)$. Then from (b)
we deduce that $\Lambda < \infty$ a.s., and from (a) we conclude that
\[ \limsup_{h \to 0+} \frac{|\tilde{v}_1(\tau_0, \Lambda+h) - \tilde{v}_1(\tau_0, \Lambda)|}{\sqrt{h^{2-\beta}\log\log(1/h)}} 
= \infty \quad \text{a.s.} \]
The proof of Lemma \ref{Lem:u_1} is complete.
\end{proof}

\smallskip

\begin{lemma}\label{Lem:u_2}
For any $\tau_0 > 0$ and $\lambda > 0$, 
\[\P\left( \limsup_{h \to 0+} \frac{|\tilde{u}_2(\tau, \lambda+h) - \tilde{u}_2(\tau, \lambda)|}{\sqrt{h^{2-\beta}\log\log(1/h)}} 
= K_\beta(\tau - \tau_0+\lambda)^{1/2} \textup{ for all } \tau \ge \tau_0 \right) = 1. \]
\end{lemma}

\begin{proof}
By Propositions \ref{Prop:sim_LIL_UB} and \ref{Prop:sim_LIL_LB}, 
\[
\P\left( \limsup_{h \to 0+} \frac{|\tilde{u}(\tau, \lambda+h) - \tilde{u}(\tau, \lambda)|}{\sqrt{h^{2-\beta}\log\log(1/h)}} 
= K_\beta (\tau+\lambda)^{1/2} \textup{ for all } \tau \ge 0 \right) = 1. 
\]
Then the result can be obtained by the observation that $\{ \tilde{u}_2(\tau_0 + \tau, \lambda), \tau, \lambda \ge 0 \}$ 
has the same distribution as $\{ \tilde{u}(\tau, \lambda), \tau, \lambda \ge 0 \}$.
Indeed, for any bounded Borel sets $A, B$ in $\R_+ \times \R$ and $c = (c_1, c_2) \in \R_+ \times \R$, by 
\eqref{Eq:noise_cov2} and change of variables we have
\begin{align*}
\E \big[W(A+c)W(B+c) \big] 
&= \int_{c_1}^\infty ds \int_\R dy \int_\R dy' \, {\bf 1}_{A}(s-c_1, y-c_2) |y-y'|^{-\beta} {\bf 1}_{B}(s-c_1, y-c_2)\\
&= \int_0^\infty ds \int_\R dy \int_\R dy' \, {\bf 1}_{A}(s, y) |y-y'|^{-\beta} {\bf 1}_{B}(s, y)\\
&=\E \big[W(A)W(B)\big].
\end{align*}
Since 
\[ \Delta \Big(\frac{\tau_0+\tau+\lambda}{\sqrt{2}}, \frac{-\tau_0-\tau+\lambda}{\sqrt{2}}\Big) \cap \Pi(\tau_0)^c %\big\{t \ge \tau_0/\sqrt{2}\big\} 
= \Delta\Big(\frac{\tau+\lambda}{\sqrt{2}}, \frac{-\tau+\lambda}{\sqrt{2}}\Big) + c,
\]
where $c = (\frac{\tau_0}{\sqrt{2}}, -\frac{\tau_0}{\sqrt{2}})$, it follows that for any $\tau, \lambda, \tau', \lambda' \ge 0$,
\begin{align*}
\E\big[\tilde{u}_2(\tau_0 + \tau, \lambda) \tilde{u}_2(\tau_0 + \tau', \lambda')\big]
&= \frac{1}{4}\, \E\bigg[W\Big(\Delta\Big(\frac{\tau_0+\tau+\lambda}{\sqrt{2}}, \frac{-\tau_0-\tau+\lambda}{\sqrt{2}}\Big) \cap \Pi(\tau_0)^c\Big)\\
&\qquad \times W\Big(\Delta\Big(\frac{\tau_0+\tau'+\lambda'}{\sqrt{2}}, \frac{-\tau_0-\tau'+\lambda'}{\sqrt{2}}\Big) \cap\Pi(\tau_0)^c \Big)\bigg]\\
& = \frac{1}{4} \, \E\bigg[ W\Big(\Delta\Big(\frac{\tau+\lambda}{\sqrt{2}}, \frac{-\tau+\lambda}{\sqrt{2}}\Big)\Big) W\Big(\Delta\Big(\frac{\tau'+\lambda'}{\sqrt{2}}, \frac{-\tau'+\lambda}{\sqrt{2}}\Big)\Big)\bigg]\\
& = \E[\tilde{u}(\tau, \lambda) \tilde{u}(\tau', \lambda')].
\end{align*}
The result follows immediately.
\end{proof}

We are now in a position to state and prove our main theorem below.
The first part shows that if we fix $\tau_0 > 0$, then based on the information from the $\sigma$-field 
$\mathscr{F}_{\tau_0}$,  we are able to show the existence of a random singularity $(\tau_0, \Lambda)$
 in the $\lambda$-direction. The second part says that if $(\tau_0, \Lambda)$ is a singular point in the 
$\lambda$-direction, then $(\tau, \Lambda)$ is also a singular point for all $\tau \ge \tau_0$.
In other words, singularities in the $\lambda$-direction propagate orthogonally, along the line that is 
parallel to the $\tau$-axis. By symmetry, it follows immediately that singularities in the $\tau$-direction 
propagate along the line parallel to the $\lambda$-axis. These are the directions of the characteristic 
lines $t+x = c$ and $t-x = c$.

\begin{theorem}\label{Thm:sing}
Let $\tau_0 > 0$. The following statements hold.
\begin{enumerate}
\item[(i)] There exists a positive, finite, $\mathscr{F}_{\tau_0}$-measurable random variable $\Lambda$ such that
\[ \limsup_{h \to 0+} \frac{|\tilde{u}(\tau_0, \Lambda + h) - \tilde{u}(\tau_0, \Lambda)|}{\sqrt{h^{2-\beta}\log\log(1/h)}} = \infty \quad \text{a.s.} \]
\item[(ii)] If $\Lambda$ is any positive, finite, $\mathscr{F}_{\tau_0}$-measurable random variable, then on an event of probability 1, we have
\[ \limsup_{h \to 0+} \frac{|\tilde{u}(\tau_0, \Lambda + h) - \tilde{u}(\tau_0, \Lambda)|}{\sqrt{h^{2-\beta}\log\log(1/h)}} 
= \infty \, \Leftrightarrow \, \limsup_{h \to 0+} \frac{|\tilde{u}(\tau, \Lambda + h) - \tilde{u}(\tau, \Lambda)|}{\sqrt{h^{2-\beta}\log\log(1/h)}} = \infty \]
for all $\tau > \tau_0$ simultaneously.
\end{enumerate}
\end{theorem}

\begin{proof}
To simplify notations, let us denote 
\[ L(\tau, \lambda) = \limsup_{h \to 0+} \frac{|\tilde{u}(\tau, \lambda+h) - \tilde{u}(\tau, \lambda)|}{\sqrt{h^{2-\beta} \log\log(1/h)}} \]
and for $i = 1, 2$,
\[ L_i(\tau, \lambda) = \limsup_{h \to 0+} \frac{|\tilde{u}_i(\tau, \lambda+h) - \tilde{u}_i(\tau, \lambda)|}{\sqrt{h^{2-\beta} \log\log(1/h)}}. \]
As in Walsh \cite{W82} and Blath and Martin \cite{BM08}, we are going to use the fact that for two functions $f$ and $g$, 
provided $\limsup_{h \to 0}|g(h)| < \infty$, we have
\begin{equation}\label{limsup_ineq}
 \limsup_{h \to 0}|f(h)| - \limsup_{h \to 0}|g(h)| \le \limsup_{h \to 0}|f(h)+g(h)|
\le \limsup_{h \to 0}|f(h)| + \limsup_{h \to 0}|g(h)|.
\end{equation}

(i). By Lemma \ref{Lem:u_1}, we can find a positive, finite, $\mathscr{F}_{\tau_0}$-measurable random variable $\Lambda$ such that 
\[ L_1(\tau_0, \Lambda) = \infty \quad \text{a.s.} \]
Since $\Lambda$ and the process $\tilde{u}_2$ are independent, Lemma \ref{Lem:u_2} implies that
\[ L_2(\tau_0, \Lambda) = K_\beta \Lambda^{1/2} \quad \text{a.s.} \]
and this is finite a.s.
Since $\tilde{u} = \tilde{u}_1 + \tilde{u}_2$, it follows from the lower bound of \eqref{limsup_ineq} that
\begin{align*}
L(\tau_0, \Lambda) \ge L_1(\tau_0, \Lambda) - L_2(\tau_0, \Lambda) = \infty \quad \text{a.s.}
\end{align*}
This completes the proof of (i).

(ii). Suppose $\Lambda$ is a positive, finite, $\mathscr{F}_{\tau_0}$-measurable random variable. By \eqref{limsup_ineq}, we have
\begin{align}\label{Ineq:L}
L_1(\tau, \Lambda) - L_2(\tau, \Lambda) \le L(\tau, \Lambda) \le L_1(\tau, \Lambda) + L_2(\tau, \Lambda)
\end{align}
for all $\tau \ge \tau_0$, provided that $L_2(\tau, \Lambda) < \infty$. Observe that for $\tau \ge \tau_0$,
\begin{equation*}\label{Eq:tilde_u1} 
\tilde{u}_1(\tau, \Lambda+h) - \tilde{u}_1(\tau, \Lambda) = \tilde{u}_1(\tau_0, \Lambda+h) - \tilde{u}_1(\tau_0, \Lambda),
\end{equation*}
hence $L_1(\tau, \Lambda) = L_1(\tau_0, \Lambda)$.
Also, by Lemma \ref{Lem:u_2} and independence between $\Lambda$ and $\tilde{u}_2$, we have
\begin{equation*}\label{Eq:tilde_u2}
\P \Big( L_2(\tau, \Lambda) = K_\beta(\tau - \tau_0 + \Lambda)^{1/2} \text{ for all } \tau \ge \tau_0 \Big) = 1.
\end{equation*}
Since $\Lambda$ is finite a.s., it follows from \eqref{Ineq:L} that
\begin{align*}
\P\Big( L_1(\tau_0, \Lambda) - K_\beta(\tau-\tau_0 + \Lambda)^{1/2} \le L(\tau, \Lambda) \le L_1(\tau_0, \Lambda) + K_\beta (\tau-\tau_0 + \Lambda)^{1/2} \text{ for all } \tau \ge \tau_0 \Big) = 1,
\end{align*}
and it implies
\[ \P\Big( L(\tau_0, \Lambda) = \infty \Leftrightarrow L(\tau, \Lambda) = \infty \text{ for all } \tau \ge \tau_0 \Big) = 1. \qedhere \]
\end{proof}

\medskip
\bigskip{\bf Acknowledgements}\,
The research of Y. Xiao is partially supported by NSF
grants DMS-1607089 and DMS-1855185.

\bigskip
\bigskip

\textsc{Cheuk Yin Lee}: Department of Statistics
and Probability, C413 Wells Hall, Michigan State University, East Lansing, MI 48824, United States

Current address: Institut de math\'ematiques, Ecole Polytechnique F\'ed\'erale de Lausanne, Station 8, 
CH-1015 Lausanne, Switzerland\\
E-mail: \texttt{cheuk.lee@epfl.ch}\\

\textsc{Yimin Xiao}: Department of Statistics
and Probability, C413 Wells Hall, Michigan State University, East Lansing, MI 48824, United States\\
E-mail: \texttt{xiao@stt.msu.edu}\\
%URL: \texttt{http://www.stt.msu.edu/\~{}xiaoyimi}

\end{document}